\documentclass{amsart}
\usepackage{amssymb,amscd,verbatim,latexsym}
\usepackage{enumerate}
\usepackage{multicol}
\usepackage[mathscr]{eucal}
\usepackage{epsfig,layout,graphics}
\newcommand{\CC}{\mathbb C}

\newcommand{\NN}{\mathbb N}

\newcommand{\RR}{\mathbb R}

\newcommand{\ZZ}{\mathbb Z}

\newcommand\pa{{\partial}}

\newcommand{\maC}{\mathcal C}

\newcommand{\maK}{\mathcal K}
\newcommand{\maL}{\mathcal L}

\newcommand{\maV}{\mathcal V}
\newcommand{\maW}{\mathcal W}
\newcommand{\maZ}{\mathcal Z}

\newtheorem{theorem}{Theorem}[section]
\newtheorem{proposition}[theorem]{Proposition}
\newtheorem{corollary}[theorem]{Corollary}

\newtheorem{lemma}[theorem]{Lemma}

\theoremstyle{definition}
\newtheorem{definition}[theorem]{Definition}
\theoremstyle{remark}
\newtheorem{remark}[theorem]{Remark}
\newtheorem{example}[theorem]{Example}

 \newcommand{\mcb}{\maC_b}

\newcommand\dom{\Omega}
\newcommand\coeff{{\beta}}
\newcommand\Wm[1]{ \breve W^{#1}}
\newcommand\Wp[1]{ W^{#1}_D }
\newcommand\ede{\, := \,}
\newcommand\seq{\, = \,}

\newcommand\Vm{{V_{m}}}
\newcommand\Vmm{{V_m^{-}}}

\author[C. B\u{a}cu\c{t}\u{a}]{Constantin B\u{a}cu\c{t}\u{a}}
\address{C. B\u{a}cu\c{t}\u{a}, University of Delaware,
        Department of Mathematical Sciences,
        501 Ewing Hall,
        Newark, DE 19716-2553, USA.}
\email{bacuta@udel.edu}

\author[H. Li]{Hengguang Li} \address{H. Li, Wayne State University,
  Department of Mathematics, Detroit, MI 48202, USA}
\email{hli@math.wayne.edu}

\author[V. Nistor]{Victor Nistor} \address{V. Nistor,
    Universit\'{e} de Lorraine, UFR MIM, Ile du Saulcy, CS 50128 57045
    METZ Cedex 01, France and V. Nistor, Pennsylvania
    State University, University Park, PA 16802, USA.}
    \email{victor.nistor@univ-lorraine.fr}

\thanks{C. Bacuta has been partially supported by NSF Grant
  DMS-1522454. H. Li has been partially supported by the NSF Grant
  DMS-1418853 and by the Wayne State University Grants Plus Program.
  V. Nistor has been partially supported by ANR-14-CE25-0012-01 (SINGSTAR).\\
Manuscripts available from {\bf http:{\scriptsize
    //}iecl.univ-lorraine.fr{\scriptsize
    /}$\tilde{}$Victor.Nistor{\scriptsize /}}
    }

\date\today
\begin{document}

\dedicatory\today

\title[Operators on domains with conical points]{
% Infinite dimensional families of 
Differential operators on domains with conical points: precise uniform 
regularity estimates}

\begin{abstract} 
 We study families of strongly elliptic, second
 order differential operators with singular coefficients on domains
 with conical points. We obtain uniform estimates on their inverses
 and on the regularity of the solutions to the associated Poisson
 problem with mixed boundary conditions. The coefficients and the
 solutions belong to (suitable) weighted Sobolev spaces.  The space of
 coefficients is a Banach space that contains, in particular, the
 space of smooth functions. Hence, our results extend classical
 well-posedness results for strongly elliptic equations in domains
 with conical points to problems with singular coefficients.  We
 furthermore provide precise uniform estimates on the norms of the
 solution operators.
\end{abstract}

\maketitle

\tableofcontents

\section{Introduction}

We consider mixed boundary value problems on a bounded,
$d$-dimensional domain $\dom$ with conical points, $d\geq 1$. The
associated differential operators belong to suitable families of
strongly elliptic, second order differential operators with singular
coefficients. We show that considering suitable singular coefficients
is natural even if one is interested only in the case of regular
coefficients. Using appropiate weighted Sobolev spaces, we obtain
uniform estimates on the norm and on the regularity of the
solutions. In addition, we provide weighted Sobolev space conditions
on the coefficients that ensure a regular dependence of the solution
on the coefficients.

To better explain our results, it is useful to put them into
perspective.  A classical result in Partial Differential Equations
states that a second order, strongly elliptic partial differential
operator $P$ induces an isomorphism
\begin{equation}\label{eq.classical}
  P \, : \, H^{m+1}(G) \cap \{u\vert_{\pa \dom} = 0 \} \,
  \stackrel{\sim}{\longrightarrow} \, H^{m-1}(G)\,,
\end{equation}
for all $m \in \ZZ_+ := \{0, 1, \ldots \}$, provided that $G$ is a
smooth, bounded domain in some euclidean space. See, for example,
\cite{ADN59, LionsMagenes1, TrudingerBook, hormander3} and the
references therein. This result has many applications and extensions.
However, it does not extend directly to non-smooth domains. In fact,
on non-smooth domains, the solution $u$ of $Pu = F$ will have
singularities, even if the right hand side $F$ is smooth. See
Kondratiev's fundamental 1967 paper \cite{Kondratiev67} for the case of
domain with conical points and Dauge's comprehensive Lecture Notes
\cite{daugeBook} for the case of polyhedral domains. See 
\cite{ApelNicaise98,BBX2003a,BBX2003b,BNZ3D1,CDN12,MR3440270,KMR,liMN10,liN09,NP,MR3300043} 
for a sample of related results. These theoretical results have been a critical 
ingredient in developing effective numerical methods approximating singular 
solutions. See for example \cite{BKP79, BZ01}. In addition, we  mention that  
estimates for equations on conical manifolds  can also  be obtained using the 
method of layer potentials (see, for example, \cite{MR2988781, MR1658089, 
MR3116171, MR1986694, LQ16} and references therein). 

For polygonal domains (and, more generally, for domains with conical
points), Kondratiev's results mentioned above extend the isomorphism
in \eqref{eq.classical} to polygonal domains by replacing the usual
Sobolev spaces $H^m(\dom)$ with the Kondratiev type Sobolev
spaces. Let $\dom$ be then a curvilinear polygonal domain (see
Definition \ref{def.polygon}, in particular, the sides are not
required to be straight), and $r_{\dom}>0$ be a smooth function on
$\dom$ that coincides with the distance to its vertices when close to
the vertices. We let
\begin{equation}\label{eq.def.Kondratiev}
 \maK_a^m(\dom) \ := \ \{ \, u : \dom \to \CC \, \vert
 \ r_{\dom}^{|\alpha| -a} \pa^{\alpha} u \in L^2(\dom), \ |\alpha|
 \leq m \, \} \,,
\end{equation}
where $\pa_i := \frac{\pa}{\pa x_i}$, $i=1, \ldots, d$, and
$\pa^{\alpha} := \pa_1^{\alpha_1}\pa_2^{\alpha_2} \dots
\pa_d^{\alpha_d}$.  Kondratiev's results \cite{Kondratiev67} (see also
\cite{CDN12, KMR}) give that the Laplacian $\Delta := \sum_{i\leq
  d}\pa_i^2$ induces an isomorphism $\Delta : \maK_{a+1}^{m+1}(\dom)
\cap \{u\vert_{\pa \dom} = 0\} \to \maK_{a-1}^{m-1}(\dom)$ for all $m
\in \ZZ_{+} = \{0, 1, \ldots \}$ and all $|a| < \pi/\alpha_{MAX}$,
where $\alpha_{MAX}$ is the maximum angle of $\dom$.  That is,
\begin{equation}
 \Delta : \maK_{a+1}^{m+1}(\dom) \cap \{u\vert_{\pa \dom} = 0\} \,
 \stackrel{\sim}{\longrightarrow} \, \maK_{a-1}^{m-1}(\dom)
\end{equation}
is a continuous bijection with continuous inverse for the indicated
values of $a$ and all non-negative integers $m$. One can extend this
result by interpolation to the usual range of values for $m$,
\cite{LionsMagenes1}.  A similar result holds also for more general
strongly elliptic operators \cite{KMR}.  In \cite{BNZ3D1}, this result
of Kondratiev was extended to three dimensional polyhedral domains and
in \cite{BMNZ} it was extended to general $d$-dimensional polyhedral
domains. In three dimensions and higher, this type of results is not
enough for numerical methods. Thus, in \cite{BNZ3D2}, an anisotropic
regularity and well-posedness result was proved for three dimensional
polyhedral domains. See also \cite{CDN12} for further references and
for related results, including analytic regularity.

In this paper, we generalize Kondratiev's result by allowing
low-regularity coefficients and by describing the dependence of the
solution on these coefficients.  To state our main result, let us fix
some notation. Let $\coeff := (a_{ij}, b_{i}, c)$ be the coefficients
of
\begin{equation}\label{eq.def.P}
 p_{\coeff} u \ede - \sum_{i,j=1}^{d} \pa_i (a_{ij} \pa_j u) +
 \sum_{i=1}^d b_{i} \pa_i u - \sum_{i=1}^{d} \pa_i (b_{d+i} u) + cu\,,
\end{equation}
a second order differential operator in divergence form on some domain
$\dom \subset \RR^{d}$.  Many concepts discussed in the paper make
sense for any dimension $d \geq 1$. Nevertheless, the main results we
prove are for $d = 2$. Thus, we assume for the rest of this
introduction that $\dom$ is a two-dimensional curvilinear polygonal
domain.  The coefficients $\coeff$ of the operator $p_{\coeff}$ are
obtained using weighted $\maW^{m, \infty}$-type space defined by
\begin{equation}\label{eq.def.maWm} 
 \maW^{m, \infty}(\dom) \ede \{\, u : \dom \to \CC \, \vert \ 
 r_{\dom}^{|\alpha|}\pa^{\alpha} u 
 \in L^{\infty}(\dom), \ |\alpha| \le m \, \}\,,
\end{equation}
where $r_{\dom}$ is as in Equation \eqref{eq.def.Kondratiev} (that is,
it is equal to the distance function to the conical points when close
to those points).  We fix for the rest of the introduction $m \in
\ZZ_{+} := \{0, 1, \ldots \}$ and we assume that $a_{ij}, r_{\dom}
b_{i}, r_{\dom}^{2} c \in \maW^{m, \infty}(\dom)$.  We let
\begin{equation}\label{eq.def.normalpha}
  \|\coeff\|_{Z_m} \ede \max\{ \| a_{ij} \|_{\maW^{m, \infty}(\dom)},
  \, \| r_{\dom} b_{i}\|_{ \, \maW^{m, \infty}(\dom)}, \, \|
  r_{\dom}^{2} c \|_{\maW^{m, \infty}(\dom)} \, \}\,,
\end{equation}
(notice the factors involving $r_{\dom}$!),\ and for $P= p_\beta$ and
$V=H^1_0(\dom)$, define
\begin{equation}\label{eq.def.rho}
 \rho(P) \ede \inf\, \frac{ \Re (Pv, v)}{\ \|v\|_{H^1(\dom)}^{2}}\, ,
 \quad v \in V \,,\ v \neq 0 \, ,
\end{equation}
where $\Re(z) = \Re z$ denotes the real part of $z$. Our main result
{\em for Dirichlet boundary conditions in two dimensions} is as
follows.

\begin{theorem}\label{thm.main} 
 Let $\dom \subset \RR^2$ be a curvilinear polygonal domain and
 $p_{\coeff} : H^1_0(\dom) \to H^{-1}(\dom)$ be as in
 \eqref{eq.def.P}.  If $\rho(\coeff) := \rho(p_{\coeff}) > 0$, then
 there exists $\eta > 0$ such that
 \begin{equation}
  p_{\coeff} \, : \, \maK_{a+1}^{m+1}(\dom) \cap \{\, u\vert_{\pa
    \dom} = 0 \, \} \, \to \, \maK_{a-1}^{m-1}(\dom)
 \end{equation}
 is an isomorphism for $|a| < \eta$ and $p_{\coeff}^{-1} :
 \maK_{a-1}^{m-1}(\dom) \to \maK_{a+1}^{m+1}(\dom) \cap \{u\vert_{\pa
   \dom} = 0\}$ depends analytically on the coefficients $\coeff :=
 (a_{ij}, b_{i}, c)$ and has norm
 \begin{equation*}
  \|p_{\coeff}^{-1}\| \, \le \, C_m (\rho(\coeff) - \gamma_1 |a| -
  \gamma_2 a^2)^{-N_m-1} \|\coeff\|_{Z_m}^{N_m}\,,
 \end{equation*}
 with $C_m, \gamma_1, \gamma_2$, and $N_m \ge 0$ independent of
 $\coeff$.
\end{theorem}

Since the solution $u$ of the equation $p_{\coeff}u = F$, $u=0$ on the
boundary, is in $\maK_{a+1}^{m+1}(\dom)$ for $F \in
\maK_{a-1}^{m-1}(\dom)$, $|a|< \eta$, we obtain the usual applications
to the Finite Element Method on straight polygonal domains for $m \ge
1$ and $a > 0$.

Theorem \ref{thm.main} is a consequence of Theorem \ref{thm.main2},
which deals with the mixed boundary value problem
\begin{equation}\label{eq.mixed}
 \begin{cases}
  \ p_{\coeff} u = f & \mbox{ in } \dom\\
  \ \ \ \ \ u = 0 & \mbox{ on } \pa_D \dom\\ 
  \ \, \pa^{\coeff}_{\nu} u = h & \mbox{ on } \pa_N \dom\,,
 \end{cases}
\end{equation}
where $(\pa^{\coeff}_{\nu}v) := \sum_{i=1}^{d} \nu_i ( \sum_{j=1}^{d}
a_{ij} \pa_j v + b_{d+i} v \big)$. An exotic example to which Theorem
\ref{thm.main2} applies is that of the Schroedinger operator $H := -
\Delta + c r_{\Omega}^{-2}$ on $\dom$ with {\em pure Neumann} boundary
conditions. The main novelties of Theorem \ref{thm.main2} (and of the 
paper in general) are the following: 
\begin{enumerate}[(i)]
 \item The precise estimate on the norm of the inverse of $p_{\coeff}$ 
 seems to be new even in the smooth case.
 \item We deal with singular coefficients of a type that has not been 
 systematically considered in the literature on non-smooth domains. 
 Thus our coefficients have both singular parts at the corners of 
 the form $r_{\Omega}^{-j}$ ($j \le 2$) and have limited regularity away 
 from the corners.
 \item We provide a new method to obtain higher regularity in weighted
 Sobolev spaces using divided differences; a method that is, in fact, 
 closer to the one used in the classical case of smooth domains. 
\end{enumerate}
 
The paper is organized as follows. In Section \ref{sec:Notation}, we
introduce the notation and necessary preliminary results for our
problem in the usual Sobolev spaces. In particular, an enhanced
Lax-Milgram Lemma (Lemma \ref{lemma.LaxMilgram}) provides uniform
estimates for the solution of our problem \eqref{eq.mixed} and
analytic dependence of this solution on the coefficients $\coeff$.  In
Section \ref{sec3}, we first define curvilinear polygonal domains
(Definition \ref{def.polygon}).  We then provide several equivalent
definitions of the weighted Sobolev spaces $\maK_{a}^{m}(\dom)$ and
the form of our differential operators.  Then, in Section \ref{sec4},
using local coordinate transformations, we derive our main result, the
analytic dependence of the solution on the coefficients in high-order
weighted Sobolev spaces (Theorem \ref{thm.main2}). Finally,
Section \ref{sec5} contains some consequences of Theorem \ref{thm.main2}
and some extensions. In particular, we consider a framework for the 
pure Neumann problem with inverse square potentials at vertices.

We thank Thomas Apel, Martin Costabel, Monique Dauge, Markus Hansen,
Serge Nicaise, and Christoph Schwab for useful comments.

\section{Coercivity in classical Sobolev spaces}
\label{sec:Notation}

In this section, we recall some needed results on coercive operators.

\subsection{Function spaces and boundary conditions}
Throughout the paper, $\dom \subset \RR^d$, $d \ge 1$, denotes a
connected, bounded domain. Further conditions on $\dom$ will be imposed
in the next section.  As usual, $H^m(\dom)$ denotes the space of
(equivalence classes of) functions on $\dom$ with $m$ derivatives in
$L^2(\dom)$. When we write $A \subset B$, we allow also $A = B$.  In
what follows, $\pa_D \dom $ is a suitable closed subset of the
boundary $\pa \dom $, where we impose {\em Dirichlet} boundary
conditions.

To formulate our problem \eqref{eq.mixed}, it is necessary to
introduce the right spaces.  We shall rely heavily on the weak
formulation of this problem.  Thus, let us recall that $H^{-1}(\dom)$
is defined as the dual space of
\begin{equation}\label{eq.def.H1}
  H^1_0(\dom ) \ede \{\, u \in H^1(\dom ) \,\vert \ u \vert_{\pa \dom
  } = 0 \, \} \,,
\end{equation}
with pivot $L^2(\dom )$. We introduce homogeneous essential boundary
conditions abstractly, by considering a subspace $V$,
\begin{equation}\label{eq.def.V}
  H^1_0(\dom ) \subset V \subset H^1(\dom ) \,,
\end{equation}
such that $V$ is a Banach space in its own topology and $H^1_0(\dom)$ is
a closed subspace of $V$. In many applications, $V$ is closed in $H^{1}(\dom)$,
but this is not the case in our application to the Neumann problem
with inverse square potentials at vertices (see Theorem \ref{thm.DN}).
Let $V^{*}$ be {\em the dual of
  $V$ with pivot space $L^2(\dom)$.} Therefore, by $(\ , \ )$ we shall
denote {\em both} the inner product $(f, g) = \int_\dom f(x)
\overline{g(x)} \,dx$ on $L^2(\dom)$, and by continuous extension,
also the duality pairing between $V^{*}$ and $V$. Thus, $V^{*} =
H^{-1}(\dom)$ if $V = H_0^1(\dom)$; otherwise, $V^{*}$ will
incorporate also non-homogeneous natural boundary conditions.

For Problem \eqref{eq.mixed}, we choose
\begin{equation}\label{eq.def.H1D}
 V \, = \, H^1_D(\dom ) \ede \{ \, u \in H^1(\dom) \, \vert \ u=0
 \mbox{ on } \pa_D \dom \, \} \,,
\end{equation}
{and assume that the Neumann part of the boundary contains no adjacent
  edges.}

\subsection{The weak formulation} 
Recall from Equation \eqref{eq.def.P} the differential operator
$p_{\coeff} u := - \displaystyle{\sum_{i,j=1}^d} \pa_i \big( a_{ij}
\pa_j u \big) + \sum_{i=1}^d b_{i} \pa_i u - \sum_{i=1}^{d}
\pa_i(b_{d+i} u) + c u,$ which is used in our problem
\eqref{eq.mixed}, where $a_{ij}, b_{i}, c: \dom \to \CC$ denote
measurable {\em complex valued} functions as in \eqref{eq.def.P} and
$\coeff$ denotes the coefficients $(a_{ij}, b_{i}, c)$. We shall make
suitable further assumptions on these coefficients below.

Equation \eqref{eq.mixed}, makes sense as formulated only if $u$ is
regular enough (at least in $H^{3/2+\epsilon}$, to validate the
Neumann derivatives at the boundary).  In order to use the Lax-Milgram
Lemma for the problem \eqref{eq.mixed}, we formulate our problem in a
more general way that allows $u \in V$. To this end, let us introduce
the Dirichlet form $B^{\coeff}$ associated to \eqref{eq.mixed}, that
is, the sesquilinear form
\begin{multline}\label{eq.def.B}
  B^{\coeff}(u,v) \, := \, \sum_{i,j=1}^d \big( a_{ij} \pa_j u, \pa_i
  v \big ) + \sum_{i=1}^d \big ( b_{i} \pa_i u, v \big ) +
  \sum_{i=1}^{d} \big (b_{d+i} u, \pa_i v \big ) + \big (c u, v \big )
  \\
  = \int_{\dom} \Big [\sum_{i=1}^d \big( \sum_{j=1}^d a_{ij}(x) \pa_j
    u(x) + b_{d+i}(x) u(x) \big ) \pa_i \overline{v(x)}
  + \Big ( \sum_{i=1}^d b_{i}(x) \pa_i u(x) + c(x) u(x) \Big )
  \overline{v(x)}\Big ] dx \, ,
\end{multline}
where $dx$ denotes the volume element in the Lebesgue integral on
$\dom \subset \RR^{d}$.

\begin{remark}\label{rem.weak}
Let $F(v)=\int_\dom f(x)v(x)dx+\int_{\pa_N\dom} h(x)v(x)dS.$ Then the
weak variational formulation of Equation \eqref{eq.mixed} is: Find $u
\in V$, such that
\begin{equation}\label{eq.weak}  
  B^{\coeff} (u, v) \ = \ F(\overline{v})\,, \ \ \mbox{ for all } \ v
  \in V \,.
\end{equation} 
We then define $P^{\coeff} : V \to V^{*}$ by
\begin{equation}\label{eq.def.weak.Pb}
 (P^{\coeff}u, v) := B^{\coeff}(u, v) \,, \ \ \mbox{ for all } \ u, v
  \in V \,.
\end{equation}
Thus, the weak formulation of Equation \eqref{eq.mixed} is equivalent
to
\begin{equation}\label{eq.problem}
 P^{\coeff} u \ = \ F \in V^{*}\,.
\end{equation}
We are interested in the dependence of $u$ on $F$ and on the
coefficients $\coeff := (a_{ij}, b_{i}, c)$ of $P^{\coeff}$.  We
notice that if the Neumann part of the boundary $\pa_N \dom$ is empty,
then $p_{\coeff}$ and $P^{\coeff}$ can be identified, but this is not
possible in general.  In fact, we are looking for an analytic
dependence of the solutions on the coefficients. For this reason, it
is useful to consider complex Banach spaces and complex valued
coefficients.
\end{remark}

\subsection{Bounded forms and operators}

For two Banach spaces $X$ and $Y$, let $\maL(X; Y)$ denote the Banach
space of continuous, linear maps $T: X \to Y$ endowed with the
operator norm
\begin{equation}
  \|T\|_{\maL(X; Y)} \ede \sup_{x \neq 0} \,
  \frac{\|Tx\|_{Y}}{\|x\|_{X}} \, .
\end{equation}
We write $\maL(X) := \maL(X; X)$.

Let us define $Z$ to be the set of coefficients $\coeff = (a_{ij},
b_{i}, c)$ such that the form $B^{\coeff}$ is defined (and continuous)
on $V \times V$, and we give $Z$ the induced norm. 
%>>
Thus $Z$ is given the induced topology from $\maL(V; V^{*})$.

\begin{corollary}\label{cor.cont}
 The map $Z \ni \coeff \to P^{\coeff} \in \maL(V; V^{*})$ is well
 defined and continuous. For each $0 < r \le \infty$, the set
\begin{equation} 
 \{ \, \coeff \in Z ,\, P^{\coeff} \mbox{ is invertible and }
 \|(P^{\coeff})^{-1}\|_{\maL(V^{*}; V)} < r \, \}
\end{equation} 
 is open in $Z$.
\end{corollary}

\begin{proof}
By the definition of $Z$, $P^{\coeff} : V \to V^{*}$ is a bounded
operator and that the map $Z \ni \coeff \to P^{\coeff} \in \maL(V;
V^{*})$ is continuous.  Next, we know that the set $\maL_{inv}(V;
V^{*})$ of invertible operators in $\maL(V; V^{*})$ is open and that
the map $P \to P^{-1}$ is continuous on $\maL_{inv}(V,
V^{*})$. Therefore the set $\{P \in \maL(V; V^{*})\, \vert
\ \|P^{-1}\|< r\}$ is open in $\maL(V; V^{*})$. Our desired set is the
inverse image of this set via the continuous map $\coeff \to
P^{\coeff}$. Since the inverse image of an open set via a continuous
map is open, the result follows.
\end{proof}

It will be convenient to use a slightly enhanced version of the
well-known Lax-Milgram Lemma stressing the analytic dependence on the
operator and on the data.  We thus first review a few basic
definitions and results on analytic functions \cite{DineenBook}.

Let $X$ and $Y$ be Banach spaces.  In what follows, $\maL_i(Y; X)$
will denote the space of continuous, multi-linear functions $L : Y
\times Y \times \ldots \times Y \to X$, where $i$ denotes the number
of copies of $Y$. The norm on the space $\maL_i(Y; X)$ is
\begin{equation*}
  \|L\|_{\maL_i(Y; X)} \ede \sup_{\|y_j\| \le 1} \|L(y_1, y_2, \ldots,
  y_i)\|_{X} \,.
\end{equation*}
Of course, $\maL_1(Y; X) = \maL(Y; X)$, isometrically.  We shall need
analytic functions defined on open subsets of a Banach space. Let $U
\subset Y$ be an open subset, then $\maC^{k}(U; X)$, $k \in \ZZ_+ \cup
\{\infty,\, \omega\}$, denotes the space of functions $v : U \to X$
with $k$ continuous (Fr\'echet) derivatives $D^i v : U \to \maL_{i}(Y;
X)$, $i\leq k$, with $D^i_a v \in \maL_i(Y; X)$ denoting the value of
$D^iv$ at $a$. Similarly, $\mcb^{k}(U; X) \subset \maC^{k}(U; X)$, $k
\in \ZZ_+ \cup \{\infty,\, \omega\}$, denotes the subspace of those
functions $v \in \maC^{k}(U; X)$ for which the derivatives $D^i v$, $i
\le k$, are bounded on $U$. For each finite $j$, we let
\begin{equation}\label{eq.def.norm}
  \|v\|_{\mcb^{j}(U; X)} \, := \, \sup_{ i \le j,\, y \in U }
  \|D^i_{y} v\|_{\maL_i(Y; X)}
\end{equation}
denotes the natural Banach space norm on $\mcb^{j}(U; X)$. 

The case $k = \omega$ refers to analytic functions, that is,
$\maC^{\omega}(U; X)$ denotes the space of functions $f : U \to X$
that have, for any $a \in U$, a uniformly convergent power series
expansion
\begin{equation*}
 f(x) \, = \, \sum_{k=0}^\infty \, \frac{1}{k!} \, D^k_{a} f(x-a, x-a,
 \ldots, x-a) \,,
\end{equation*}
for $x$ in a small, non-empty open ball centered at $a$.  If $k$ is
not finite, that is, if $k = \infty$ or $k = \omega$, we endow
$\mcb^{k}(U; X)$ with the Fr\'echet topology defined by the family of
seminorms $\| \, \cdot \, \|_{\mcb^{j}(U; X)}$, $j \ge 1$.  We shall
need the following standard result.

\begin{lemma} \label{lemma.analytic} Let $Y_1, Y_2$ be Banach
  spaces.
\begin{enumerate}[(i)]
\item The map $\maL(Y_1; Y_2) \times Y_1 \ni (T, y) \to T y \in Y_2$
  is analytic.
\item The map $T\to T^{-1}\in\maL(Y_1)$ is analytic on the open set
  $\maL_{\operatorname{inv}}(Y_1)$ of invertible operators in
  $\maL(Y_1) := \maL(Y_1; Y_1)$.
\end{enumerate}
\end{lemma}

\begin{proof}
 In (i), the desired map is bilinear, and hence analytic. To prove
 (ii), we simply write the Neumann series formula $(T-R)^{-1} =
 \sum_{n=0}^{\infty} T^{-1}(RT^{-1})^n$, which is uniformly and
 absolutely convergent for $\|R\| \|T^{-1}\| \le 1-\epsilon$,
 $\epsilon > 0$.
\end{proof}

\subsection{An enhanced Lax-Milgram Lemma}
We now recall the classical Lax-Milgram Lemma, in the form that we
will need.

\begin{definition}\label{def.coercive}
Let $H_0^1(\dom) \subset V \subset H^1(\dom)$.  A continuous operator
$P: V \to V^{*}$ is called {\em coercive on $V$} (or simply {\em
  coercive} when there is no danger of confusion) if
\begin{equation*}
 0 < \rho(P) \ede \inf_{v \in V\smallsetminus \{0\}} \, \frac{\Re (Pv,
   v)}{\ \|v\|_{V}^{2}} \, .
\end{equation*} 
\end{definition}

We shall usually write $\rho(\coeff) = \rho(P^{\coeff})$, where
$\rho(P^{\coeff})$ is as defined in Equation (\ref{eq.def.rho}).  For
$P = P^{\coeff}$, we thus have $\rho(\coeff) \|v\|_{H^{1}(\dom )}^2 =
\rho(P^{\coeff}) \|v\|_{H^{1}(\dom )}^2 \le \Re B^{\coeff}(v,v)$, for
all $v \in V$.  We shall need the following simple observation:

\begin{remark}\label{rem.LaxMilgram}
If $P : V \to V^{*}$ is coercive {on $V$} and $P_1 : V \to V^{*}$
satisfies $\|P_1\| < \rho(P)$, then $P + P_1$ is also coercive on $V$
and $\rho(P + P_1) \ge \rho(P) - \|P_1\|$. Indeed,
\begin{equation}\label{eq.pert.coerc}
 \Re \big ( (P + P_1) u, u \big ) \, \ge \, \Re (P u, u) - \|P_1\|
 \|u\|_{V}^2 \, \ge \, (\rho(P) - \|P_1\|)
 \|u\|_{V}^2\,,
\end{equation}
and hence the set $\maL(V; V^{*})_c$ of coercive operators is open in
$\maL(V; V^{*})$.
\end{remark}

Recall now the standard way of solving Equation \eqref{eq.weak} using
the Lax-Milgram Lemma for coercive operators.

\begin{lemma}[Analytic Lax-Milgram Lemma]\label{lemma.LaxMilgram}
Assume that $P : V \to V^{*}$ is coercive. Then $P$ is invertible and
$\|P^{-1}\| \le \rho(P)^{-1}$.  Moreover, the map $\maL(V; V^{*})_{c}
\times V^{*} \ni (P, F) \to P^{-1}F \in V$ is analytic. Consequently,
 \begin{equation*}
  Z \cap \maL(V; V^{*}) \times V^{*} \, \ni \, (\coeff, F) \, \to \,
  (P^{\coeff})^{-1}F \, \in \, V
 \end{equation*}
 is analytic as well.
\end{lemma}

\begin{proof}
 The first part is just the classical Lax-Milgram Lemma
 \cite{BrezisBookFA, Ciarlet78, MonkBook}, which states that
 ``coercivity implies invertibility'' and gives the norm estimate.
 The second part follows from Lemma \ref{lemma.analytic}. Indeed, the
 map $\Phi : \maL(V; V^{*})_{c} \times V^{*} \to V$, $\Phi (\coeff, F) :=
 (P^{\coeff})^{-1}F$ is the composition of the maps
\begin{equation*}
  \begin{gathered}
   \maL(V; V^{*})_{c} \times V^{*} \times V^{*} \, \ni\, (\coeff, F)
 \,\to\, (P^{\coeff}, F) \,\in\, \maL_{inv}(V, V^{*}) \times V^{*}\,,\\
 \maL_{inv}(V; V^{*}) \times V^{*} \,\ni\, (P, F) \,\to\, (P^{-1}, F)
 \,\in\, \maL(V^{*}; V) \times V^{*}\,, \ \mbox{ and }\\
 \maL(V^{*}; V) \times V^{*} \,\ni\, (P^{-1}, F) \,\to\, P^{-1}F
 \,\in\, V \,.
  \end{gathered}
\end{equation*}
 The first of these three maps is well defined and linear by the
 classical Lax-Milgram Lemma.  The other two maps are analytic by
 Lemma \ref{lemma.analytic}.  Since the composition of analytic
 functions is analytic, the result follows.
\end{proof}

Examples of coercive operators are obtained using ``uniformly strongly
elliptic'' operators, whose definition we recall next.

\begin{definition}\label{def.use}
Let $\coeff \in {Z} $. The operator $P^{\coeff}$ is called {\em
  uniformly strongly elliptic} if there exists $C > 0$ such that
\begin{equation}\label{eq.unif.strong}
   \sum_{ij=1}^{d} \, \Re \big ( a_{ij}(x) \xi_i \xi_j \big) \, \ge \,
   C \|\xi\|^2 \,,
\end{equation}
for all $\xi = (\xi_i) \in \RR^d$ and all $x \in
\overline{\dom}$. Here $\| \cdot \|$ denote the standard euclidean
norm on $\RR^{d}$. The largest $C$ with the property in
\eqref{eq.unif.strong} will be denoted $C_{use}(\coeff)$.
\end{definition}

Then, we have the following standard example.

\begin{example}
Let $\coeff \in {Z} $, as in Definition \ref{def.use}.  We shall
regard a matrix $X := [x_{ij}]$, $(X)_{ij} = x_{ij}$, as a linear
operator acting on $\CC^{d}$ by the formula $X \zeta = \xi$, where
$\xi_{i} = \sum_{j} x_{ij}\zeta_{j}$. We consider the adjoint and
positivity with respect to the usual inner product on $\CC^{d}$.  We
thus have $X \ge 0$ if, and only if $(X \xi, \xi) = \sum_{ij}
x_{ij}\xi_j \overline{\xi_i} \ge 0$ for all $\xi \in \CC^{d}$. Also,
recall that $X^{*}$, the adjoint of the matrix $X$, has entries
$(X^*)_{ij} = \overline{x_{ji}}$. Then $P^{\coeff}$ is uniformly
strongly elliptic if, and only if, there exists $\gamma > 0$ such that
the matrix $a(x) := [a_{ij}(x)]$ of highest order coefficients of
$P^{\coeff}$ satisfies
\begin{equation}\label{eq.c.invert}
  a(x) + a(x)^* \, \ge \, \gamma I_{d}\,, \quad \mbox{for all } x \in
  \dom \,,
\end{equation}
where $I_{d}$ denotes the unit matrix on $\CC^{d}$.  Assume also that
$b_{i} = c = 0$. Then,
\begin{multline*}
  2 \Re (P^{\coeff} u, u) \ede 2 \Re\Big(\int_{\dom} \,
  {\sum_{i,j=1}^d} a_{ij}(x) \pa_j u(x) \pa_i \overline{u(x)}\,
  dx\Big) \, = \, 2 \Re (a \nabla u, \nabla u) \\
  = \, (a \nabla u, \nabla u) + ( \nabla u, a \nabla u) \, = \, \big
  ((a+a^*) \nabla u, \nabla u \big ) \, \geq \, \gamma \| \nabla
  u\|_{L^2(\dom)}^2 \,.
\end{multline*}
If, moreover, $\pa_D \dom$ has positive measure, then there exists $c
= c_{\dom, \pa_D \dom} > 0$ such that $ \int_{\dom} |\nabla v|^2 dx
\ge c \|v\|_{H^1(\dom)}^2$ for all $v \in H_D^1(\dom)$, and hence
$P^{\coeff}$ is coercive on $V = H^1_D(\dom)$.  (Recall that
$H^1_D(\dom)$ was defined in Equation \eqref{eq.def.H1D}.  In
particular, $v = 0$ on $\pa_D \dom$ if $v \in H^1_D(\dom)$.)
\end{example}

We then have the following standard result, whose proof we recall for
the benefit of the reader. (See also \cite{CDS2011,NistorSchwab1}.)

\begin{proposition}\label{prop.pos.ell}
If $\coeff = (a_{ij}, b_{i}, c) \in {Z} $ is such that $P^{\coeff}$ is
coercive on $H^{1}_0(\dom) \subset V \subset H^{1}(\dom)$, then
$P^{\coeff}$ is strongly elliptic, more precisely, the estimate
\eqref{eq.unif.strong} is satisfied for any $C \le \rho(\coeff) :=
\rho(P^{\coeff})$. Moreover, $P^{\coeff} : V \to V^{*}$ is a
continuous bijection and $(P^{\coeff})^{-1}F$ depends analytically on
the coefficients $\coeff$ and on $F \in V^{*}$.
\end{proposition}

\begin{proof} The second part is an immediate consequence of the 
analytic Lax-Milgram Lemma. Let us concentrate then on the first part.
Let us assume that $P$ is coercive and let $\xi = (\xi_i) \in \RR^d$.
Also, let us choose an arbitrary smooth function $\phi$ with compact
support $D$ in $\dom$. We then define the function $\psi \in
\maC_c^{\infty}(\dom) \subset V$ by the formula $\psi(x) := e^{\imath
  t \xi \cdot x} \phi(x) \in \CC,$ where $\imath := \sqrt{-1}$ and
$\xi \cdot x = \sum_{k=1}^d \xi_k x_k$.  Then $\pa_j \psi (x) = \imath
t \xi_j e^{\imath t \xi \cdot x} \phi(x) + e^{\imath t \xi \cdot x}
\pa_j\phi(x)$, and hence $\imath t \xi_j e^{\imath t \xi \cdot x}
\phi(x)$ is the dominant term in $\pa_j \psi (x)$ as $t \to
\infty$. Taking into account all the indices $j$ and computing the
squares of the $L^2$-norms, we obtain
\begin{equation} \label{eqn.one}
  \lim_{t \to \infty} t^{-2} \|\psi\|_{H^{1}(D)}^2 \, = \,
  \sum_{j=1}^d \xi_j^2 \int_{D} |\phi(x)|^2 \,dx \, = \, \|\xi\|^2
  \int_{D} |\phi(x)|^2 \, dx \,.
\end{equation}
Similarly, the coefficients $a_{ij}$ of $P^{\coeff}$, are estimated
using ``oscillatory testing''
\begin{equation}\label{eqn.two}
  \lim_{t \to \infty} t^{-2} (P^{\coeff} \psi, \psi) \ = \ \int_{D}
  \ \sum_{i,j=1}^d a_{ij}(x, y) \xi_i \xi_j |\phi(x)|^2 \, dx\,.
\end{equation}
We then use Definition \ref{def.coercive} for $v = \psi$ and we pass
to the limit as $t \to \infty$. By coercivity and the definition of
$\rho(\coeff) := \rho(P^{\coeff})$, we have that $\rho(\coeff)
\|\psi\|_{H^{1}(D)}^2 \le \Re (P^{\coeff}\psi, \psi)$. Dividing this
inequality by $t^{-2}$ and taking the limit as $t \to \infty$, we
obtain from Equations \eqref{eqn.one} and \eqref{eqn.two} that
\begin{equation*}
     \rho(\coeff) \|\xi\|^2 \int_{D} |\phi(x)|^2 \,dx \, \le \, \Re\,
     \int_{D} \sum_{ij} a_{ij}(x, y) \xi_i \overline\xi_j |\phi(x)|^2
     \,dx \,.
\end{equation*}
Since $\phi$ is an arbitrary compactly supported smooth function on
$D$, it follows that, for all $x\in D$,
\begin{equation*}
    \rho(\coeff) \|\xi\|^2 \, \le \, \Re \, \sum_{ij} a_{ij}(x) \xi_i
    \xi_j \:.
\end{equation*}
Since $\xi$ is arbitrary, we obtain Equation \eqref{eq.unif.strong}
with $C = \rho(P)$.
\end{proof}

An immediate corollary of Proposition \ref{prop.pos.ell} is

\begin{corollary}\label{cor.pos.elliptic}
 We have $\rho(P) \le C_{use}$.
\end{corollary}

This inequality will be used in the form $C_{use}^{-1} \le
\rho(P)^{-1}$ in the following sections.

\section{Polygonal domains, operators, and weighted Sobolev spaces}\label{sec3}

In this section, we introduce the domains, the weighted Sobolev
spaces, and the differential operators that we shall use. We also
provide several equivalent definitions of the weighted Sobolev spaces
and prove some intermediate results.

\subsection{Polygonal domains and defining local coordinates} 
In this section, we let $\dom$ be a {\em curvilinear polygonal
  domain,} although our method works without significant change for
domains with conical points.

Let us describe in detail our domain $\dom$ as a Dauge-type corner
domain, with the purpose of fixing the notation and of introducing
some useful local coordinate systems -- called ``defining
coordinates'' -- that will be used in the proofs below. Let $B_j$
denote the open unit ball in $\RR^{j}$. Thus $B_0$ is reduced to one
point, $B_1 = (-1, 1)$, and $B_2 = \{(x, y) \subset \RR^2, x^2 + y^2 <
1\}$.

\begin{definition}\label{def.polygon}
 A {\em curvilinear polygonal domain} $\dom \subset \RR^2$ is an open,
 bounded subset of $\RR^{2}$ with the property that for every point $p
 \in \overline{\dom}$ there exists $j \in \{0, 1, 2\}$, a neighborhood
 $U_p$ of $p$ in $\RR^2$, and a smooth map $\phi_p : \RR^2 \to \RR^2$
 that defines a diffeomorphism $\phi_p : U_p \to B_{j} \times B_{2-j}
 \subset \RR^{2}$, $\phi_p(p) = 0$, satisfying the following
 conditions:
 \begin{enumerate}[(i)]
  \item If $j=2$, then $U_p \subset \dom$;
  \item If $j=1$, then $\phi_p(U_p \cap \dom) = B_1 \times (0, 1)$ or
    $\phi_p(U_p \cap \dom) = B_1 \times (B_1 \smallsetminus \{0\})$;
  \item If $j =0$, then $\phi_p(U_p \cap \dom) = \{(r\cos \theta , r
    \sin \theta) | \text{ with } r \in (0, 1), \theta \in I_p\}$, for
    some finite union $I_p$ of open intervals in $S^{1}$.
 \end{enumerate}
 For $p \in \overline{\dom}$, we let $j_p$ the {\em largest} $j$ for
 which $p$ satisfies one of the above properties.
\end{definition}

These are essentially the {\em corner domains} in
\cite{daugeBook}. The definition above was generalized to arbitrary
dimensions in \cite{BMNZ}. See also \cite{KMR, liMN10, elast, NP}. The
second case in (ii) corresponds to cracks in the domain.  We continue
with some remarks.

\begin{remark} We notice that in the two cases (i) and (iii) of Definition 
\ref{def.polygon} ($j=2$ and $j=0$), the spaces $\phi_p(U_p) = B_j
\times B_{2-j}$ will be the same (up to a canonical diffeomorphism),
but the spaces $\phi_p(U_p \cap \dom)$ will not be diffeomorphic.
\end{remark}

\begin{remark}
Let $\dom$ be a curvilinear polygonal domain and $p \in
\overline{\dom}$.  Then $p$ satisfies the conditions of the definition
for {\em exactly} one value of $j$, {\em except} the case when $p$ is
on a smooth part of the boundary, when a choice of $j = 1$ or $j=0$ is
possible. This is the case exactly when $j_p = 1$.  If $j = 0$ is
chosen, then $I_p$ is half a circle.
\end{remark}

\begin{remark}
The set $\maV_g := \{p \in \overline{\dom} \vert \, j_p = 0\}$ is
finite and is contained in the boundary of $\dom$. It is the set of
{\em geometric vertices}.
\end{remark}

Let us choose for each point $p \in \overline{\dom}$ a value $j = i_p$
that satisfies the conditions of the definition. If $j_p = 1$, we
choose $i_p = j_p = 1$, except possibly for finitely many points $p
\in \overline{\dom}$.  These points will be called {\em artificial
  vertices}.  The set of all vertices (geometric and artificial) is
finite, which will be denoted by $\maV$, and {\em will be fixed in
  what follows}. We assume that all points where the boundary
conditions change are in $\maV$. We also fix the resulting polar
coordinates $r \circ \phi_p$ and $\theta \circ \phi_p$ on $U_p$, for
all $p \in \maV$.

\begin{definition}\label{def.can.chart}
 The coordinate charts $\phi_p : U_p \to B_{j} \times B_{2-j}$ of
 Definition \ref{def.polygon} that were chosen such that $j = i_p$ are
 called the {\em defining} coordinate charts of the curvilinear
 polygonal domain $\dom$.  (Recall that $j = i_p = 0$ if, and only if
 $p \in \maV$.)
\end{definition}

\begin{remark}
Artificial vertices are useful, for instance, in the case when we have
a change in boundary conditions or if there are point singularities in
the coefficients, see \cite{liMN10, liN09} and the references therein.
The right framework is, of course, that of a stratified space
\cite{BMNZ}, with $j_p$ denoting the dimension of the stratum to which
$p$ belongs, but we do not need this in the simple case at hand.
\end{remark}

\begin{remark}
It follows from Definition \ref{def.polygon} that if $\dom$ is a
curvilinear polygonal domain, then the set $\pa\dom \smallsetminus
\maV$ is the union of finitely many smooth, open curves $e_j : (-1, 1)
\to \pa \dom$. The curves $e_j$ have as image the {\em open edges} of
$\pa \dom$ and we shall sometimes identify $e_j$ with its image. The
curves $e_j$ are disjoint and have no self-intersections. The closure
of (the image of $e_j$) is called a {\em closed edge}. {Thus, the
  vertices are not contained in the open edges (but they are, of
  course, contained in the closed edges).} Our assumption that all
points where the boundary conditions change are in $\maV$ implies that
$\pa_D \dom$ consists of a union of closed edges of $\dom$.
\end{remark}

\subsection{Equivalent definitions of weighted spaces}
In this section, we discuss some equivalent definitions of weighted
Sobolev spaces. We adapt to our setting the results in \cite{sobolev},
to which we refer for more details.

We shall fix, from now on, a {\em finite} set of defining coordinate
charts $\phi_{k} = \phi_{p_k}$, for some $p_k \in \dom$, $1 \le k \le
N$, so that $U_k := U_{p_k}$, $1 \le k \le N$, defines a finite
covering of $\overline{\dom}$. Thus, for $p=p_k$ such that $j_p \neq
0$, the coordinates are $(x, y) \in \RR^{2}$. Otherwise, these
coordinates will be denoted by $(r, \theta) \in (0, 1) \times
S^{1}$. We may relabel these points such that $p_k$ is a vertex if,
and only if, $1 \le k \le N_0$. We then have the following alternative
definition of the weighted Sobolev spaces $\maK_a^m(\dom)$. We denote
\begin{equation}\label{eq.def.vf}
 X_k u \ede \pa_x u \ \ \mbox{ and } \ \ Y_k u \ede \pa_y u\,, 
 \quad \mbox{for } N_0 < k \le N\,,
\end{equation}
in the coordinate system defined by $\phi_k = \phi_{p_k} = (x, y) \in
\RR^{2}$ that corresponds to one of the chosen points $p_k$, provided
that is {\em not} a vertex.  If, however, $p_k$ is a vertex, then we
let
\begin{equation}\label{eq.def.vf2}
 X_k u \ede r \pa_r u \ \ \mbox{ and } \ \ Y_k u \ede \pa_\theta u\,,
 \quad \mbox{for } 1 \le k \le N_0\,,
\end{equation}
in the coordinate system defined by $\phi_k = (r, \theta) \in (0, 1)
\times S^{1}$.  Note the appearance of $r$ in front of $\pa_r$!

\begin{remark}\label{rem.loc.coeff}
 Assuming that the coefficients are locally Lipschitz, we can express
 the differential operator $r_{\dom}^2 p_{\coeff}$ in any of the
 coordinate systems $\phi_k : U_k \to \RR^{2}$. That means that, for
 each $1 \le k \le N$, we can find coefficients $c, c_1, c_2, c_{11},
 c_{12}, c_{22}$ such that
 \begin{equation}\label{eq.coord.syst}
   p_{\coeff}u \seq (c_{11} X_k^2 + c_{12} X_kY_k + c_{22} Y_k^2 + c_1
   X_k + c_2 Y_k + c)u \quad \mbox{ on } U_k \,,
 \end{equation}
 with the vector fields $X_k$ and $Y_k$ introduced in Equations
 \eqref{eq.def.vf} and \eqref{eq.def.vf2}.
\end{remark}

For each open subset $U \subset \dom$, let us denote 
\begin{equation}\label{eq.def.norm.U}
 \|u\|_{\maK_a^{m}(U)}^{2} \ede \sum_{|\alpha| \le m}
 \|r_{\dom}^{|\alpha|-a} \pa^{\alpha} u\|_{L^2(U)}^{2}\,.
\end{equation}
Thus, if $U = \dom$, $\|u\|_{\maK_a^{m}(U)} =
\|u\|_{\maK_a^{m}(\dom)}$ is simply the norm on
$\maK_a^{m}(\dom)$. Note that the weight $r_{\dom}$ is not intrinsic
to the set $U$, but depends on $\dom$, which is nevertheless not
indicated in the notation, in order not to overburden it.  We define
the spaces $\maW^{m, \infty}(U)$ similarly as in (\ref{eq.def.maWm})
with the same weight $r_{\dom}$.  We then have the following result
that, in particular, provides an alternative definition of the
weighted Sobolev spaces $\maK_{a}^{m}(\dom)$ introduced in Equation
\eqref{eq.def.Kondratiev}.

\begin{proposition}\label{prop.alt.def}
 Let $u : \dom \to \CC$ be a measurable function and $U \subset \dom$
 and open subset. We have that $u \in \maK_{a}^{m}(U)$ if, and only
 if, $r_{\dom}^{-a} X_k^iY_k^ju \in L^2(U\cap U_{k})$, where $U_k =
 U_{p_k}$, for all $k$ and all $i + j \le m$. Moreover, the
 $\maK_{a}^{m}(U)$--norm is equivalent to the norm
 \begin{equation*}
  ||| u |||'_{U} \ede \sum_{k=1}^{N} \sum_{i+j \le m} 
  \| r_{\dom}^{-a} X_k^{i} Y_k^{j} u \|_{L^{2}(U \cap U_k)}\,.
 \end{equation*}
\end{proposition}

\begin{proof}
 This follows right away from the definition of the
 $\maK_{a}^{m}(U)$-norm. Indeed, away from the vertices, both the
 $|||\, \cdot \,|||_{U}'$-norm and the $\maK_{a}^{m}$-norm coincide
 with the usual $H^{m}$-norm. On the other hand, near a vertex, or
 more generally on an angle $\Xi := \{(r, \theta)\vert\, \alpha <
 \theta < \beta \}$, both norms are given by $\{u\vert\, r^{-a} (r
 \pa_r)^i \pa_\theta^j u \in L^{2}(\Xi)\}$.  For the
 $\maK_{a}^{m}(U)$-norm this is seen by writing $\pa_x$ and $\pa_y$
 in polar coordinates, more precisely, from
 \begin{equation}\label{eq.polar.coord}
  r \pa_x \seq (\cos \theta) r \pa_r - (\sin \theta) \pa_\theta \quad
  \mbox{ and } \quad r \pa_y \seq (\sin \theta) r \pa_r + (\cos
  \theta) \pa_\theta\,.
 \end{equation}
 See \cite{ammannNistorWReg, liN09} for more details.
\end{proof}

We finally have the following corollary.

\begin{corollary}\label{cor.norm}
 The norm $\|u\|_{\maK_{a}^{m+1}(\dom)}$ is equivalent to the norm
 \begin{equation*}
  |||u||| \ede \|u\|_{\maK_{a}^{m}(\dom)} + \sum_{k=1}^{N} \big (
  \|X_k u\|_{\maK_{a}^{m}(U_k)} + \|Y_k u\|_{\maK_{a}^{m}(U_k)} \big
  )\,.
 \end{equation*}
\end{corollary}

\begin{proof}
In the definition of $|||u|||'$, Proposition \ref{prop.alt.def}, with
$m$ replaced by $m+1$, we collect all the terms with $i + j \le m$ and
notice that they give the norm for $\maK_{a}^{m}$. The rest of the
terms will contain at least one differential $X_k$ or one differential
$Y_k$ and thus are of the form $\| r_{\dom}^{-a} X_k^{i} Y_k^{j} Y_ku
\|_{L^{2}(\dom)}$ or $\| r_{\dom}^{-a} X_k^{i} Y_k^{j} X_ku
\|_{L^{2}(\dom)}$, $ i + j \le m$, since the differential operators
$X_k$ and $Y_k$ commute on $U_k$.
\end{proof}

\subsection{The differential operators}
We include in this subsection the definition of our differential
operators and three needed intermediate results (lemmas).

We introduce now our set of coefficients.  Recall the norm
$\|\coeff\|_{Z_m}$ introduced in Equation \eqref{eq.def.normalpha} and
let
\begin{equation}\label{eq.def.Zm}
 Z_m \ede \{ \, \coeff = (a_{ij}, b_{i}, c) \, \vert \ \mbox{such that
 } \|\coeff\|_{Z_m} < \infty \, \}\,.
\end{equation}

Note that for example, the Schroedinger operator $-\Delta + r^{-2}$ is
an operator of the form $P^{\coeff}$ for suitable $\coeff \in Z_m$.

Below, we shall often use inequalities of the form $A \le C B$, where
$A$ and $B$ are expressions involving $u$ and $\coeff$ and $C \in
\RR$. We shall say that $C$ is an {\em admissible bound} if it does
not depend on $u$ and $\coeff$, and then we shall write $ A \le_c B$.

\begin{lemma}\label{lemma.Cuse} 
Let $\coeff = (a_{ij}, b_{i}, c) \in Z_m$, $m \ge1$, and let us
express $p_{\coeff}$ as in Remark \ref{rem.loc.coeff}. Then $c, c_1,
c_2, c_{11}, c_{12}, c_{22} \in \maW^{m -1, \infty}(U_k)$. Moreover,
 \begin{equation*}
  \|c\|_{\maW^{m-1, \infty}(U_k)} + \|c_1\|_{\maW^{m-1, \infty}(U_k)}
  + \ldots + \|c_{22} \|_{\maW^{m-1, \infty}(U_k)} \ \le_c
  \ \|\coeff\|_{Z_m}\,.
 \end{equation*}
 If $p_{\coeff}$ is moreover uniformly strongly elliptic, then
 $|c_{22}^{-1}| \le_c C_{use}^{-1}$ on $U_k$.
\end{lemma}

\begin{proof} 
We first notice that since $m \ge 1$, we can convert our operator to a
non-divergence form operator. Indeed, one can simply replace a term of
the form $\pa_i a \pa_j u$ with $a \pa_i \pa_j u + (\pa_i a) \pa_j u$,
where $u \in \maK^{m+1}_{a+1}(\dom)$ and $r_{\Omega} \pa_i a \in
\maW^{m-1, \infty}(\dom)$.  We deal similarly with the terms of the
form $\pa_i (b_{i} u)$.  This accounts for the loss of one derivative
in the regularity of the coefficients of $c, \ldots , c_{22}$.

We need to show that the coefficients $c, \ldots, c_{22}$ are in
$\maW^{m-1, \infty}(\dom)(U_k)$ and that they have the indicated
bounds. To this end, we consider the two possible cases: when $U_k$
contains no vertices of $\dom$ (equivalently, if $k > N_0$) and the
case when $U_k$ is centered at a vertex.

If $k > N_0$, then the coefficients $c, \ldots, c_{22}$ can be
expressed using $\phi_k$ and its derivatives linearly in terms of the
coefficients $\coeff$ on the closure of $U_k$. Since there is a finite
number of such neighborhoods and $\phi_k$ and its derivatives are
bounded on the closure of $U_k$, the bound for the coefficients $c,
\ldots , c_{22}$ in terms of $\|\coeff\|_{Z_m}$ on $U_k$ follows using
a compactness argument.  In particular, the bound $|c_{22}^{-1}| \le_c
C_{use}^{-1}$ follows from the uniform ellipticity of $p_{\coeff}$ on
$\overline{U_k}$.

If, on the other hand, $k \le N_0$ (that is, $U_k$ is centered at a
vertex). Let us concentrate on the highest order terms, for
simplicity.  We then have, up to lower order terms (denoted $l.o.t$)
\begin{align*}
 r^{2} \pa_x^2 & = (\cos \theta)^2 (r \pa_r)^2 - 2 (\sin \theta \cos
 \theta ) r \pa_r \pa_\theta + (\sin \theta)^2 \pa_\theta^2 +
 l.o.t. \\
 r^{2} \pa_x \pa_y & = (\sin \theta \cos \theta) (r \pa_r)^2 + (\cos^2
 \theta - \sin^2 \theta) r \pa_r \pa_\theta + (\sin \theta \cos
 \theta) \pa_\theta^2 + l.o.t. \\
 r^{2} \pa_y^2 & = (\sin \theta)^2 (r \pa_r)^2 + 2 (\sin \theta \cos
 \theta ) r \pa_r \pa_\theta + (\cos \theta)^2 \pa_\theta^2 + l.o.t.
\end{align*}
The bound on the coefficients $c, \ldots, c_{22}$ follows since $\sin
\theta$ and $\cos \theta$ are in $\maW^{m, \infty}(U_k)$ for all $m$.
This gives also that $c_{22} = a_{11} \cos^2 \theta + 2 a_{12} \cos
\theta \sin \theta + a_{22} \sin^2 \theta \ge C_{use}$ for the
coefficient $c_{22}$ of $Y_k^2 = \pa_{\theta}^2$.  (Thus
$|c_{22}^{-1}| \le C_{use}^{-1}$ on $U_k$, for $k \le N_0$.)
\end{proof}

For instance, for the Laplacian in polar coordinates, we have
\begin{equation*}
 r_{\dom}^{2} \Delta \, = \, (r \pa_r)^2 + \pa_\theta^2 \, = \, X_k^2
 + Y_k^2
\end{equation*}
in the neighborhood $U_k$ of the vertex $p_k$. 

The following lemma will be used in the proof of Theorem
\ref{thm.main2} and explains some of the calculations there.

\begin{lemma}\label{lemma.prod} For two functions $b$ and $c$, 
 we have
 \begin{enumerate}[(i)]
 \item $\|bc\|_{\maK_{a}^{m}(\dom)} \le C \|b\|_{\maW^{m,
     \infty}(\dom)} \|c\|_{\maK_{a}^{m}(\dom)}$.
  \item $\|bc\|_{\maW^{m, \infty}(\dom)} \le C \|b\|_{\maW^{m,
      \infty}(\dom)} \|c\|_{\maW^{m, \infty}(\dom)}$, therefore
    $\maW^{m, \infty}(\dom)$ is an algebra.
 \item If $b\in \maW^{m, \infty}(\dom)$ is such that $b^{-1} \in
   L^{\infty}(\dom) = \maW^{0, \infty}(\dom)$, then $b$ is invertible
   in $\maW^{m, \infty}(\dom)$ and
 \begin{equation*}
  \|b^{-1}\|_{\maW^{m, \infty}(\dom)} \ \le \ C \,
  \|b^{-1}\|_{L^{\infty}(\dom)}^{m+1} \, \|b\|_{\maW^{m,
      \infty}(\dom)}^{m} \, .
 \end{equation*}
 \end{enumerate}
 The parameter $C$ depends only on $m$ and $\dom$.
\end{lemma}

\begin{proof}
 This is a direct calculation. Indeed, the first two relations are
 based on the rule $\pa^{\alpha}(bc) = \sum_{\beta \le \alpha}\,
 {\alpha \choose \beta}\, \pa^{\beta} b \, \pa^{\alpha - \beta}
 c$. The last one is obtained from the relation $\pa^{\alpha}(b^{-1})
 = b^{-1-|\alpha|}Q$, where $Q = Q(b, \pa_1b, \pa_2b, \ldots,
 \pa^{\alpha}b)$ is a polynomial of degree $|\alpha|$ in all
 derivatives $\pa^{\beta}b$, with $0\leq \beta \leq \alpha$.  This
 relation is proved by induction on $|\alpha|$.
\end{proof}

For further reference, we shall need the following version of
``Nirenberg's trick,'' (see, for instance, \cite{ADN59,
  TrudingerBook}).

\begin{lemma}\label{lemma.Nirenberg}
 Let $T : X \to Y$ be a continuous, bijective operator between two
 Banach spaces $X$ and $Y$. Let $S_X(t)$ and $S_Y(t)$ be two $c_0$
 semi-groups of operators on $X$, respectively $Y$, with generators
 denoted by $A_X$ and, respectively, $A_Y$.  We assume that for any $t
 > 0$, there exists $T_t \in \maL(X; Y)$ such that $S_Y(t) T = T_t
 S_X(t)$. Assume that $t^{-1}(T_t - T)$ converges strongly as $t \to
 0$ to a bounded operator $B$. Then $T$ maps bijectively the domain of
 $A_X$ to the domain of $A_Y$ and we have that $A_X T^{-1} \xi =
 T^{-1} \big ( A_Y \xi - B T^{-1} \xi \big)$, for all $\xi$ in the
 domain of $A_Y$. Consequently,
 \begin{equation*}
  \|A_X T^{-1} \xi \|_{X} \ \le \ \|T^{-1}\| \, \big ( \, \|A_Y
  \xi\|_{Y} + \|B\| \| T^{-1} \xi\|_{X} \, \big)\,.
 \end{equation*}
\end{lemma}

\begin{proof}
 We have that $\xi \in X$ is in the domain of $A_X$, the generator of
 $S_X$ if, and only if, the limit $A_X \xi := \lim_{t \to 0}
 t^{-1}\big ( S_X(t) - 1) \xi$ exists. The definition of $T_t$ gives
 \begin{equation*}
  t^{-1}\big ( S_Y(t) -1) T \xi \ = \ t^{-1}\big (T_t - T)S_X(t) \xi
  \, + \, t^{-1} T \big ( S_X(t) - 1)\xi\,.
 \end{equation*}
 Since $t^{-1}(T_t - T)\zeta \to B \zeta$ for all vectors $\zeta \in
 X$ and $B : X \to Y$ is bounded, we obtain that the limit $\lim_{t
   \to 0}t^{-1}\big ( S_Y(t) -1) T \xi$ exists if, and only if, the
 limit $\lim_{t \to 0}t^{-1}\big ( S_X(t) -1) \xi$ exists.  This shows
 that $T$ maps bijectively the domain of $A_X$ to the domain of $A_Y$
 and that $A_Y T = B + T A_X$. Multiplying by $T^{-1}$ to the left and
 to the right gives the desired result.
\end{proof}

One can use Lemma \ref{lemma.Nirenberg} as a regularity estimate.

\section{Higher regularity in weighted Sobolev spaces}\label{sec4}

In this section, we prove our main result, Theorem \ref{thm.main2}.
Theorem \ref{thm.main} is an immediate consequence of this theorem and
of Remark \ref{rem.a.neq.0}.

\subsection{The higher regularity problem}
We now come back to the study of our mixed problem, as formulated in
Equation \eqref{eq.mixed}.  We are interested in solutions with more
regularity than the ones provided by the space $V$ appearing in its
weak formulation, Equation \eqref{eq.weak} or Equation
\eqref{eq.problem}. While for the weak formulation the classical
Sobolev spaces suffice, the higher regularity is formulated in the
framework of the weighted Sobolev spaces considered by Kondratiev
\cite{Kondratiev67} and others, see also \cite{CostabelStephan85, dauge1982}.

We thus introduce 
\begin{equation}\label{eq.def.Vm-1}
\begin{gathered}
 \Vm(a) \ede \maK_{a+1}^{m+1}(\dom) \, {\cap \, \{u\vert_{\pa_D \dom}
   = 0 \} } \ \mbox{ for } \ \ m \in\ZZ_{+} = \{ 0, 1, 2, \ldots \}
 \ \ \mbox{ and}\\
 \Vmm(a) \ede \maK_{a-1}^{m-1}(\dom) \oplus \maK_{a-1/2}^{m-1/2}(\pa_N
 \dom) \ \ \mbox{ for } \ \ m \in \NN = \{ 1, 2, \ldots \}. \ \ \\
\end{gathered}
\end{equation}
The spaces $\maK_{a-1/2}^{m-1/2}(\pa_N \dom)$, $m \ge 1$, are the
spaces of traces of functions in $\maK_{a}^{m}(\dom)$, in the sense
that the restriction at the boundary defines a continuous, surjective
map $ \maK_{a}^{m}(\dom) \to \maK_{a-1/2}^{m-1/2}(\pa_N \dom)$
\cite{ammannNistorWReg}.  The space $\maK_{a}^{m}(\pa_N \dom)$ can be
defined directly for $m \in \ZZ_+$ in a manner completely analogous to
the usual Kondratiev spaces. For non-integer regularity, they can be
obtained by interpolation, \cite{sobolev, ammannNistorWReg}.

We recall that differentiation defines continuous maps $\pa_j :
\maK_{a}^{m}(\dom) \to \maK_{a-1}^{m-1}(\dom)$. In the same way, the
combination of the normal derivative at the boundary
$(\pa^{\coeff}_{\nu}v) := \sum_{i=1}^{d} \nu_i ( \sum_{j=1}^{d}a_{ij}
\pa_j v + b_{d+i} v \big)$ and restriction at the boundary define a
continuous, surjective map $\pa^{\coeff}_\nu : \maK_{a}^{m}( \dom) \to
\maK_{a-3/2}^{m-3/2}(\pa_N \dom)$, $m \ge 2$.

\begin{lemma}\label{lemma.cont3} We have continuous maps
\begin{equation}\label{eq.def.Pma}
\begin{gathered}
 P^{\coeff}(m, a) \ede (p_{\coeff}, \pa^{\coeff}_\nu) \, : \, \Vm(a)
 \, \to \, \Vmm(a)\,, \quad m \ge 1\,,\\
 P^{\coeff}(m, a)(u) \seq \Big ( \, \sum_{ij} \pa_i ( a_{ij} \pa_j u)
 + \sum_{i} b_i \pa_i u + cu, \sum_{ij} \nu_i a_{ij} \pa_j
 u\vert_{\pa_N \dom} \, \Big )\,.
\end{gathered}
\end{equation}
\end{lemma}

Therefore the operators $P^{\coeff}(m, a)$, $m \in \NN$, $a \in \RR$,
are given by the same formula (but have different domains and ranges).

\begin{remark}
Let us assume for this remark that $a = 0$ and discuss this case in
more detail. If $\pa_N \dom$ contains no adjacent edges, the the Hardy
inequality \cite{BNZ3D1,KMR} shows that the natural inclusion
\begin{equation}\label{eq.inclusion}
  \maK_1^1(\dom) \cap \{ u\vert_{\pa_D \dom} = 0 \} \, \to \,
  H^1_D(\dom) \ede H^1(\dom) \cap \{ u\vert_{\pa_D \dom} = 0 \}
\end{equation}
is an isomorphism (that is, it is continuous with continuous inverse).
We thus consider $V := V_0(0)$ in general (for all $\pa_N \dom$).  For
symmetry, we also let $V_0^{-}(0) := V^{*}$ and
\begin{equation}\label{eq.def.Pma0}
 P^{\coeff}(0, 0) \ede P^{\coeff} \, : \, V_0(0) \seq V \to
 V_{0}^{-}(0) \ede V^{*}\,,
\end{equation}
which is, of course, nothing but the operator studied before.

We then have 
\begin{equation*}
 V_{m+1}(0) \subset \Vm(0) \quad \mbox{and} \quad V_{m+1}^{-}(0)
 \subset \Vmm(0)\, \quad \mbox{for all } m \ge 0\,.
\end{equation*}
This is trivially true for $m >0$. For $m=0$, in which case we need to
construct the natural inclusion $\Phi : V^{-}_{m}(0) \to
V_{0}^{-}(0)$, $m\geq 1$. The map $\Phi$ associates to $(f, h) \in
V^{-}_{m}(0) := \maK_{-1}^{m-1}(\dom) \oplus \maK_{-1/2}^{m-1/2}(\pa_N
\dom)$ the linear functional $F:=\Phi(f, h)$ on $V$, $F \in V^{*}$
defined by the formula
\begin{equation}
 F(v) \, = \, \Phi(f, h)(v) \ede \int_{\dom} f \, v\, dx + \int_{\pa_N
   \dom} h \, v \, dS \,,
\end{equation}
where $dx$ is the volume element on $\dom$ and $dS$ is the surface
element on $\pa \dom$. With this definition of the inclusion $\Phi :
V^{-}_{m}(0) \to V_0^{-}(0) := V^*$, we obtain that $P^{\coeff}(m, 0)$
is the restriction of $P^{\coeff}(0, 0)$ to $\Vm(0)$. In other words,
we have the commutative diagram
\begin{equation}\label{restriction}
\begin{CD}
\Vm(0) @>{P^{\coeff}(m, 0)}>> \Vmm(0) \\
@VVV @VVV \\
V_0(0) := V @>{P^{\coeff}(0, 0) := P^{\coeff}}  >> V^{-}_{0}(0)
\end{CD}
\end{equation}
with the operators $P^{\coeff}$ introduced in Equations 
\eqref{eq.def.Pma} and \eqref{eq.def.Pma0}.
\end{remark}

See also Remark \ref{rem.weak}. We now return to the general case $a
\in \RR$.

\begin{remark}\label{rem.a.neq.0}
We then notice that we have
\begin{equation*}
  \Vm(a) \seq r_{\dom}^{a} \Vm(0) \ \ \mbox{ for } \ \ m \ge 0
  \ \ \mbox{ and }\ \ \Vmm(a) \seq r_{\dom}^{a} V_m^{-}(0)\ \ \mbox{
    for } m > 0\,.
\end{equation*}  
We then let
\begin{equation*}
 V_{0}^{-}(a) \ede r_{\dom}^{a} V_0^{-}(0) \seq r_{\dom}^{a} V^{*}\,.
\end{equation*}
By symmetry, we obtain
\begin{equation}\label{eq.inclusions}
 V_{m+1}(a) \subset \Vm(a) \quad \mbox{and} \quad V_{m+1}^{-}(a)
 \subset \Vmm(a)\, \quad \mbox{for all } m \ge 0\,,
\end{equation}
in general (for all $a$). In fact, the relation between the spaces
above for different values of $a$ allows us to reduce to the case
$a=0$ since, if $\coeff \in Z_m$, then there exists $\coeff(a) \in
Z_m$ such that
\begin{equation}\label{eq.rel.Ps}
 P^{\coeff}(m, a) \seq r^{a}_{\dom} P^{\coeff(a)}(m, 0) r^{-a}_{\dom}
 \,, \quad m \ge 1\,.
\end{equation}
This can be seen from $r^{a} (\pa_j) r^{-a} u = \pa_j u - a x_jr^{-1}
u$ and $x_j/r \in \maW^{m, \infty}(\dom)$ for all $m$.  In particular,
$\coeff(a) = \coeff + a \gamma_1 + a^{2} \gamma_2$, with $\gamma_1,
\gamma_2 \in Z_m$, whenever $\coeff \in Z_m$.  (This explains why it
is crucial to consider coefficients in weighted spaces of the form
$\maW^{m, \infty}(\dom)$ as well as in terms of the form $\pa_i (b_i
u)$ in the definition of $p_{\coeff}$.) We use Equation
\eqref{eq.rel.Ps} to define $P^{\coeff}(0, a)$ for all $a$.  Of
course, $P^{\coeff}(0, 0) = P^{\coeff} : V \to V^{*}$.
\end{remark}

Our {\em higher regularity problem} is then to establish conditions
for $P^{\coeff}(m, a)$ to be an isomorphism, which is achieved in
Theorem \ref{thm.main2}.

\subsection{Extension of Theorem \ref{thm.main} and its proof}

For its proof, it will be convenient to extend the differential
operators $X_k$, $Y_k$ from $U_k$ to the whole domain $\dom$. We
choose these extensions so that
\begin{enumerate}[(i)]
 \item If $p_k$ is a vertex, then all $X_j$, $Y_j$, $j \neq k$, vanish
   in a neighborhood of $p_k$.
 \item For all $k$, $X_k$ (regarded as a vector field) is tangent to
   all edges (if $X_k$ vanishes at a point on an edge, it is
   considered to be tangent to the edge at that point).
\end{enumerate}

Recall that $\rho(P) := \inf_{v \neq 0}\, \Re (Pv,
v)/\|v\|_{V}$, for any linear map $P : V \to V^{*}$, that
$\rho(\coeff) := \rho(P^{\coeff})$, and that $C_{use}^{-1} \le
\rho(\coeff)^{-1}$.

\begin{theorem}\label{thm.main2} 
 Let $\dom \subset \RR^2$ be a bounded, curvilinear polygonal domain
 and $\coeff = (a_{ij}, b_{i}, c) \in Z_m$. If $P^{\coeff(a)} : V \to
 V^{*}$ is {coercive}, then $P^{\coeff}(m, a) : \Vm(a) \to \Vmm(a)$ is
 invertible for all $m \ge 0$.  Moreover, $P^{\coeff}(m, a)^{-1}$
 depends analytically on the coefficients $\coeff$ and there exists
 $C_m = C_m(\dom, \pa_D \dom)$ and $N_m \ge 0$ such that
 \begin{equation*}
  \|P^{\coeff}(m, a)^{-1} \|_{\maL(\Vmm; \Vm)} \ 
  \le  \ C_m \left (\rho(\coeff(a))\right )^{-N_m-1} \, 
  \|\coeff(a)\|_{Z_m}^{N_m} \,.
 \end{equation*}
\end{theorem}

\begin{proof}
In view of Remark \ref{rem.a.neq.0} and of the relation in Equation
\eqref{eq.rel.Ps}, we can reduce the proof of this theorem to the case
$a=0$.  Because of this, {\em we shall assume for the rest of this
  section that $a = 0$ and we shall write $\Vm(0) = \Vm$ and $\Vmm(0)
  = \Vmm$.}  We also denote $\|(P^{\coeff})^{-1}\|_{m} :=
\|(P^{\coeff})^{-1} \|_{\maL(\Vmm; \Vm)}$.

 For $m = 0$, we can just take $C_0 = 1$ and $N_0 = 0$ and then the
 result reduces to the Lax-Milgram Lemma \ref{lemma.LaxMilgram}.  In
 general, we adapt to our setting the classical method based on finite
 differences (see for example \cite{TrudingerBook, LionsMagenes1,
   CDNBook}), which was used in similar settings in \cite{BNZ3D1,
   BNZ3D2, HNParam, NistorSchwab1}. We thus give a summary of the
 argument. For simplicity, we drop $\dom$ from the notations of the
 norms. In this proof, as throughout the paper, $C$ is a parameter
 that is independent of $\coeff$ or $F$, and hence it depends only on
 $\dom$, $\pa_N \dom$, $m$, and the choice of the vector fields $X_k$
 and $Y_k$ (and of their initial domains $U_k$). However, we write $A
 \le_c B$ instead of $A \le CB$, if $C$ is such a bound.

 Let us notice that since $\|P^{\coeff}(m, 0)\| \|P^{\coeff}(m,
 0)^{-1}\| \ge \|P^{\coeff}(m, 0)P^{\coeff}(m, 0)^{-1}\| \ge 1$ and
 since $\|P^{\coeff}(m, 0) \|_{m} \le_c \|\coeff\|_{Z_m}$ we have that
 \begin{equation*}
  \|\coeff\|_{Z_m} \|(P^{\coeff})^{-1}\|_{m} \ge 1/C > 0 \,. 
 \end{equation*}
 When $m=0$, we also have $\rho(\coeff)^{-1} \ge \|(P^{\coeff})^{-1}\|
 =:\|(P^{\coeff})^{-1}\|_{0}$, and hence
 \begin{equation}\label{eq.beta.P}
  R(\coeff) \ede \|\coeff\|_{Z_m} \rho(\coeff)^{-1} \, 
  \ge \, \|\coeff\|_{W^{0, \infty}} \|(P^{\coeff})^{-1}\|_{0} \ge 1/C > 0 \,. 
 \end{equation}
 
 To show that the operator $P^{\coeff}(m, 0) : \Vm(0) \to \Vmm(0)$ is
 invertible and to obtain estimates on $\|(P^{\coeff})^{-1}\|_{m} :=
 \|P^{\coeff}(m, 0)\|$, we proceed by induction on $m$. As we have
 explained above, for $m = 0$, this has already been proved. We thus
 assume that $P^{\coeff}(m-1, 0)$ is invertible and that it satisfies
 the required estimate, which we write as
 \begin{equation*}
  \|(P^{\coeff})^{-1}\|_{m-1} \ede \|P^{\coeff}(m-1, 0)^{-1}
  \|_{\maL(\Vmm; \Vm)} \ \le \ C_{m-1}
  \frac{R(\coeff)^{N_{m-1}}}{\rho(\coeff)} \,.
 \end{equation*}

 Let $F \in \Vmm$ be arbitrary but fixed. We know by the induction
 hypothesis that $u:= (P^{\coeff})^{-1}F = P^{\coeff}(m-1, 0)^{-1}F
 \in V_{m-1}$, but we need to show that it is in fact in $\Vm $ and to
 estimate its norm in terms of $\|F\|_{\Vmm}$.  Since $\Vm :=
 \maK_{1}^{m+1} \cap V$, it is enough to show that $u \in
 \maK_{1}^{m+1}$ and to estimate $\|u\|_{\maK^{m+1}_{1}} =
 \|(P^{\coeff})^{-1}F\|_{\maK^{m+1}_{1}}$.
 
 First of all, by Corollary \ref{cor.norm}, it is enough to estimate
 $\|X_k u\|_{\maK_1^{m} }$ and $\|Y_k u\|_{\maK_1^{m} }$. Indeed,
 \begin{equation}\label{eq.zero}
  \|u\|_{\maK^{m+1}_1} \, \le_c \, \, \|u\|_{\maK^{m}_1} +
  \sum_{k=1}^N \|X_k u\|_{\maK^{m}_1(U_k)} + \sum_{k=1}^{N} \|Y_k
  u\|_{\maK^{m}_1(U_k)} \,,
 \end{equation}
 and the first term on the right hand side is estimated by induction
 on $m$ by
 \begin{equation}\label{eq.unu}
  \|u\|_{\maK^{m}_1} \ \, \le \, \frac{C_{m-1}
    R(\coeff)^{N_{m-1}}}{\rho(\coeff)}\|F\|_{V_{m-1}^{-}} \, \le \,
  \frac{C_{m-1} R(\coeff)^{N_{m-1}}}{\rho(\coeff)}\|F\|_{\Vmm} \,.
 \end{equation}
 (Note that the other terms in Equation \eqref{eq.zero} are computed
 on smaller subsets $U_k$.)

 Let us estimate now the other terms in the sum appearing on the right
 hand side of the inequation \eqref{eq.zero}.  First, since $X_k$ is
 tangent to all edges of $\dom$, it integrates to a one parameter
 family of diffeomorphisms of $\dom$, and hence to strongly continuous
 one-parameter groups of continuous operators on $X := V_{m-1}$ and $Y
 := V_{m-1}^{-}$, due to the particular form of boundary conditions
 used to define these spaces. Let us denote by $S_X(t) : X \to X$ and
 $S_Y(t) : Y \to Y$, $t \in \RR$, the operators defining these
 one-parameter groups of operators.  We have that
 \begin{equation*}
  B \ede X_k P^{\coeff} - P^{\coeff} X_k \seq \lim_{t \to 0}
  t^{-1}(S_X(t) P^{\coeff}S_Y(-t) - P^{\coeff}) \seq P^{\coeff'} \,,
 \end{equation*}
 and hence $\coeff' \in Z_{m-1}$ is obtained by taking derivatives of
 $\coeff$.  Therefore $B : X \to Y$ is bounded by Lemma
 \ref{lemma.cont3}.  The assumptions of Lemma \ref{lemma.Nirenberg}
 are therefore satisfied.  Moreover, $\|B\| \le_c\|\coeff'\|_{Z_{m-1}}
 \le_c\|\coeff\|_{Z_m}$, which allows us to conclude that
 \begin{equation}\label{eq.one}
  \| X_k u \|_{\maK_1^{m} } \ \le_c \ \|(P^{\coeff})^{-1}\|_{m-1} \big
  ( \|X_k F\|_{V_{m-1}^{-} } + \|\coeff\|_{Z_m}
  \|(P^{\coeff})^{-1}\|_{m-1} \|F\|_{V_{m-1}^{-} } \big ) \,.
 \end{equation}
 Using also the relation $\|\coeff\|_{Z_m} \|P^{\coeff}(m-1, 0)^{-1}\|
 \ge 1/C$ of Equation \eqref{eq.beta.P}, we obtain
 \begin{eqnarray}
  \| X_k u \|_{\maK_1^{m} } &\le_c& \|(P^{\coeff})^{-1}\|_{m-1} \big
  (1 + \|(P^{\coeff})^{-1}\|_{m-1} \|\coeff\|_{Z_m} \big )
  \|F\|_{\Vmm}\nonumber\\ & \le_c & \frac{R(\coeff)^{2N_{m-1} +
      1}}{\rho(\coeff)} \label{eq.doi} \|F\|_{\Vmm}.
 \end{eqnarray}
 
 We now turn to the study of the terms $\| Y_k u \|_{\maK_1^{m}}$, for
 which we need to use the strong ellipticity of $P^{\coeff}$ (as in
 the classical methods \cite{LionsMagenes1, TrudingerBook}) together
 with Lemmas \ref{lemma.Cuse} and \ref{lemma.prod}. First of all,
 Lemma \ref{lemma.Cuse} provides us with the decomposition $c_kY_k^2 u
 = r_{\dom}^2 P^{\coeff} u - Q_ku$, where $c_k \in \maW^{m,
   \infty}(U_k)$ and $Q_k$ is a sum of differential operators of the
 form $Y_kX_k$ and $X_k^2$ and lower order differential operators
 generated by $X_k$ and $Y_k$ with coefficients in $\maW^{m,
   \infty}(U_k)$.  This gives using first the general form of the
 $\|\, \cdot \, \|_{\maK_1^{m} (U_k) }$-norm
 \begin{multline}\label{eq.two}
  \| Y_k u \|_{\maK_1^{m} (U_k) } \ \le_c \ \|Y_k u \|_{\maK_1^{m-1}}
  \, + \, \|X_k Y_k u \|_{\maK_1^{m-1} (U_k)} + \|Y_k^2 u
  \|_{\maK_1^{m-1}(U_k)} \\
\ \le_c \ \|u \|_{\maK_1^{m}} \, + \, \| Y_k X_k u
\|_{\maK_1^{m-1}(U_k)} + \| Y_k^2 u \|_{\maK_1^{m-1}(U_k)} \\
  \ \le_c \ \|u \|_{\maK_1^{m}} \, + \| X_k u \|_{\maK_1^{m}} \, + \,
  \| c_{k}^{-1} ( r_{\dom}^2p_{\coeff} - Q_{k})u
  \|_{\maK_1^{m-1}(U_k)}\,.
\end{multline}
The first term in the last line of Equation \eqref{eq.two} is
estimated by the induction hypothesis in Equation \eqref{eq.unu}. The
second one is estimated in Equation \eqref{eq.doi}. To estimate the
third term, we obtain directly from Lemma \ref{lemma.Cuse} the
following
\begin{enumerate}
\item each $c_{k} \in \maW^{m, \infty}(U_k)$ is bounded in terms of
  $\|\coeff\|_{Z_m}$,
 
\item the coefficients of $X_k^2$, $X_k Y_k$, $X_k$, and $Y_k$ and the
  free term of $Q_{k}$ (which is no longer in divergence form) are in
  $\maW^{m-1, \infty}(U_k)$ and are also bounded in terms of
  $\|\coeff\|_{Z_m }$,
 
\item $\|c_{k}^{-1}\|_{L^{\infty}} \le_c C_{use}^{-1} \le_c
  \rho(\coeff)^{-1}$.
\end{enumerate}
 Hence 
\begin{multline}\label{eq.ck}
 \|c_{k}^{-1}\|_{\maW^{m-1, \infty}(U_k)} \, \le_c \,
 \|c_{k}^{-1}\|_{L^{\infty}(U_k)}^{m} \|c_{k}\|_{\maW^{m-1,
     \infty}(U_k)}^{m-1} \\
  \, \le_c \, \rho(\coeff)^{-m} \|\coeff \|_{\maW^{m, \infty}}^{m-1}
  \seq \rho(\coeff)^{-1} R(\coeff)^{m-1} \, ,
\end{multline}
where the first inequality is by Lemma \ref{lemma.prod}(iii).

We have, successively
\begin{equation}\label{eq.Pbeta}
  \| r_{\dom}^2 p_{\coeff}u \|_{\maK_1^{m-1}(U_k)} \, \le_c \,
  \|p_{\coeff}u\|_{\maK_{-1}^{m-1}(U_k)} \, \le_c \,
  \|p_{\coeff}u\|_{\maK_{-1}^{m-1}} \, \le_c \, \|F
  \|_{V_{m-1}^{-}}\,.
\end{equation}
Similarly, let $\nu$ be the $\maW^{m-1, \infty}(U_k)$ norm of the
coefficients of $Q_k$, then $\nu \le_c \|\coeff\|_{\maW^{m, \infty}}$
and hence
\begin{multline}\label{eq.Qk}
 \| Q_{k} u \|_{\maK_1^{m-1}(U_k)} \, \le_c \, \nu \, \Big ( \|X_k^2
 u\|_{\maK^{m-1}_1(U_k)} + \|Y_k X_k u\|_{\maK^{m-1}_1(U_k)} \\
  + \|X_k u\|_{\maK^{m-1}_1(U_k)} + \|Y_k u\|_{\maK^{m-1}_1(U_k)} +
  \|u\|_{\maK^{m-1}_1(U_k)} \Big )\\
 \le_c \, \|\coeff\|_{Z_m} \Big ( \|X_k u\|_{\maK^{m}_1} +
 \|u\|_{\maK^{m}_1} \Big )\\
 \le_c \, ( R(\coeff)^{N_{m-1} +1} + R(\coeff)^{2N_{m-1} + 2} )\|F
 \|_{V_{m-1}^{-}}\, \le_c \, R(\coeff)^{2N_{m-1} + 2} \|F
 \|_{V_{m-1}^{-}}\,,
\end{multline}
where we have used also Equations \eqref{eq.unu} and \eqref{eq.doi}.
Consequently,
\begin{multline}\label{eq.Y2.c}
 \| c_{k}^{-1} ( r_{\dom}^2p_{\coeff} - Q_{k})u \|_{\maK_1^{m-1}(U_k)}
 \, \le_c \, \| c_{k}^{-1}\|_{\maW^{m-1, \infty}} \|
 r_{\dom}^2p_{\coeff}u - Q_{k}u \|_{\maK_1^{m-1}(U_k)}\\
 \le_c \, \frac{R(\coeff)^{m-1}}{\rho(\coeff)} \big (1 +
 R(\coeff)^{2N_{m-1} + 2}) \|F \|_{V_{m-1}^{-}} \le_c \,
 \frac{R(\coeff)^{2N_{m-1} + m +1}}{\rho(\coeff)} \|F
 \|_{V_{m-1}^{-}}.
\end{multline}
Substituting back into Equation \eqref{eq.two} the estimates of
Equations \eqref{eq.unu}, \eqref{eq.doi}, and \eqref{eq.Y2.c}, we
obtain
\begin{multline}\label{eq.trei}
  \| Y_k u \|_{\maK_1^{m} (U_k) } \ \le_c \ \frac{ R(\coeff)^{N_{m-1}}
    + R(\coeff)^{2N_{m-1} + 1} + R(\coeff)^{2N_{m-1} + m +1}
  }{\rho(\coeff)}\|F\|_{\Vmm}\\
  \le_c \ \frac{ R(\coeff)^{2N_{m-1} + m +1}
  }{\rho(\coeff)}\|F\|_{\Vmm}.
\end{multline}

In a completely analogous manner, substituting back into Equation
\eqref{eq.zero} the estimates of Equations \eqref{eq.unu},
\eqref{eq.doi}, and \eqref{eq.trei}, we obtain

\begin{multline}\label{eq.final}
  \|u\|_{\maK^{m+1}_1} \, \le_c \, \, \|u\|_{\maK^{m}_1} +
  \sum_{k=1}^N \|X_k u\|_{\maK^{m}_1(U_k)} + \sum_{k=1}^{N} \|Y_k
  u\|_{\maK^{m}_1(U_k)} \\
  \le_c \ \frac{ R(\coeff)^{N_{m-1}} + R(\coeff)^{2N_{m-1} + 1} +
    R(\coeff)^{2N_{m-1} + m +1} }{\rho(\coeff)}\|F\|_{\Vmm}\\
  \le_c \ \frac{ R(\coeff)^{2N_{m-1} + m +1}
  }{\rho(\coeff)}\|F\|_{\Vmm}\,.
\end{multline}
In all the statements above, saying $\|v\|_Z < \infty$ for some Banach
space $Z$ means, implicitly, that $v \in Z$.  We thus have that $u \in
\maK^{m+1}_1$ and that it satisfies the required estimate with $N_m =
2N_{m-1} + m + 1$.  The proof is complete.
\end{proof}

\begin{remark}
 Remark \ref{rem.a.neq.0} gives that there exist parameters $\gamma_1$
 and $\gamma_2$, independent of $\coeff$, such that $\rho(\coeff(a))
 \ge \rho(\coeff) - \gamma_1 |a| - \gamma_2 a^{2}$.  Moreover, an
 induction argument gives that $N_m =2^{m+2}-m-3 \ge 0$ in two
 dimensions. We ignore if this is true in higher dimensions as well.
\end{remark}

As mentioned in the introduction, an example to Theorem
\ref{thm.main2} is the Schroedinger operator $H := - \Delta + c
r_{\Omega}^{-2}$ on $\dom$ with {\em pure Neumann} boundary
conditions. See also Theorem \ref{thm.DN}.

\section{Extensions and applications\label{sec5}}

\subsection{Some direct consequences of Theorem \ref{thm.main2}}

We conclude with a few  corollaries. For simplicity, we formulate them
only in the case $a =  0$, since Remark \ref{rem.a.neq.0} allows us to
reduce to the  case $a = 0$.  Throughout the rest of  this section, we
continue to assume that $\coeff = (a_{ij}, b_{i}, c) \in Z_m$ and that
$\dom$ is  a bounded, curvilinear  polygonal domain with  $\pa_D \dom$
nonempty.

Recall that $\maL(V; V^{*})_{c} \subset Z$ denotes the set of coefficients 
that yield a coercive operator.

\begin{corollary}\label{cor.two} 
Let $U := \maL(V; V^{*})_{c} \cap Z_m$. Then $U$ is an open subset of
$Z_m$ and the map $U \times \Vmm \ni (\coeff, F) \to
(P^{\coeff})^{-1}F \in \Vm$ is analytic and
\begin{equation*}
 \|(P^{\coeff})^{-1} F\|_{\Vm} \ \le \ C_m \,
 \frac{\|\coeff\|_{Z_m}^{N_m}}{\rho(\coeff)^{N_m + 1}} \,
 \|F\|_{\Vmm} \,.
\end{equation*}
\end{corollary}

\begin{proof}
 The inclusion $\maW^{m, \infty}(\dom) \to Z_m$ is continuous and
 $\maL(V; V^{*})_c \cap Z_{m}$ is open in $Z_m $.  Hence $U$ is open
 in $Z_m$. Next we proceed as in Lemma \ref{lemma.LaxMilgram} using
 that the map $\Phi : U \times \Vmm \to \Vm$, $\Phi (\coeff, F) :=
 (P^{\coeff})^{-1}F$ is the composition of the maps
 \begin{equation*}
  \begin{gathered}
   U \times \Vmm \, \ni\, (\coeff, F) \,\to\, (P^{\coeff}, F) \,\in\,
   \maL_{inv}(\Vm; \Vmm) \times \Vmm\,,\\
 \maL_{inv}(\Vm; \Vmm) \times \Vmm \,\ni\, (P, F) \,\to\, (P^{-1}, F)
 \,\in\, \maL(\Vmm; \Vm) \times \Vmm\,, \ \mbox{ and }\\
 \maL(\Vmm; \Vm) \times \Vmm \,\ni\, (P^{-1}, F) \,\to\, P^{-1}F
 \,\in\, \Vm \,.
  \end{gathered}
 \end{equation*}
 The first of these three maps is well defined and linear by Theorem
 \ref{thm.main2}.  The other two maps are analytic by Lemma
 \ref{lemma.analytic}. Since the composition of analytic functions is
 analytic, the result follows.
\end{proof}

The following result is useful in approximating solutions of
parametric problems.

\begin{corollary}\label{cor.three} Let $Y$ be a Banach
space and let $U \subset Y$ be an open subset. Let $F : U \to \Vmm$
and $\coeff : U \to \maL(V; V^{*})_{c} \cap \maW^{m, \infty}(\dom)$ be
analytic functions. Then $U \ni y \to (P^{\coeff(y)})^{-1}F(y) \in
\Vm$ is analytic and we have
\begin{equation*}
 \|(P^{\coeff(y)})^{-1} F(y)\|_{\Vm} \ \le \ C_m \,
 \frac{\|\coeff(y)\|_{Z_m}^{N_m}}{\rho(\coeff(y))^{N_m + 1}} \,
 \|F(y)\|_{\Vmm} \,.
\end{equation*}
In particular, if the functions $\|\coeff(y)\| =
\|\coeff(y)\|_{\maW^{m,\infty}(\dom)}$ and $\|F(y)\|_{\Vmm}$ are
bounded {\em and} there exists $c > 0$ such that $\rho(\coeff(y)) >
c$, then $(P^{\coeff(y)})^{-1} F(y)$ is a bounded analytic function.
\end{corollary}

\begin{proof}
 The composition of two analytic functions is analytic. The first part
 is therefore an immediate consequence of the first part of Corollary
 \ref{cor.two}. The second part follows also from Corollary
 \ref{cor.two}.
\end{proof}

%\subsection{General analytic dependence on coefficients and regularity}

The method used to obtain analytic dependence of the solution in terms
of coefficients can be extended to other settings.

\begin{remark}
Let us assume the following:
 \begin{enumerate}[(i)]
 \item We are given continuously embedded Banach spaces
  $\Wp{m+1} \subset V \subset H^{1}(\dom)$, $\Wm{m-1} \subset V^{*}$, and $\maZ \subset Z_m$
   satisfying the following properties:
\item For any $\coeff \in \maZ$, the operator $P^{\coeff}$ defines continuous maps
 $V \to V^{*}$ and $\Wp{m+1} \to \Wm{m-1}$.
\item 
  $\|P^{\coeff}\|_{\maL(\Wp{m+1}; \Wm{m-1})} \le_c \|\coeff\|_{\maZ}$
  and $\|P^{\coeff}\|_{\maL(V; V^{*})} \le_c \|\coeff\|_{\maZ}$.
\item If $\coeff \in \maZ$ and $P^{\coeff} : V \to V^{*}$ is coercive,
  then the map $(P^{\coeff})^{-1} : V^{*} \to V $ maps $\Wm{m-1}$ to
  $\Wp{m+1}$ continuously and there exists a continuous, increasing function
  $\mathfrak{N}_m : \RR_{+}^2 \to \RR_{+}$ such that
\begin{equation*}
  \| ( P^{\coeff})^{-1}\|_{\maL(\Wm{m-1}; \Wp{m+1})} \, \le \,
  \mathfrak{N}_m\big ( \rho(\coeff)^{-1} \, , \, \|\coeff\|_Z \big )
  \,.
\end{equation*}
\end{enumerate}
Then our previous results (in particular, Corollaries \ref{cor.two}
and \ref{cor.three}) extend to the new setting by replacing
$\maW^{m,\infty}(\dom)$ with $Z$, $\Vm$ with $\Wp{m+1}$, $\Vmm$
with $\Wm{m-1}$, and by using $\mathfrak{N}_m$ in the bounds
for the norm. We thank Markus Hansen and Christoph
Schwab for their input related to this remark.
\end{remark}

\subsection{General domains with conical points}

The same argument as in the proof of Theorem \ref{thm.main2} gives a
proof of a similar result on general domains with conical points. In
the neighborhood of a conical point, the domain is of the form $\dom =
\{ r x' \vert \ 0 < r < 1, x' \in \omega$, where $\omega \subset
S^{n-1}$ is a smooth domain on the unit sphere $S^{n-1}$. The main
difference is that we will need to additionally straighten the
boundary of $\omega$.

\subsection{Dirichlet and Neumann boundary conditions} 
We conclude this paper by an application of Theorem \ref{thm.main2} to
estimates for Schroedinger operators.  We note that the following
result applies to arbitrary mixed boundary conditions (including pure
Neumann).

\begin{theorem}\label{thm.DN} Let 
$P^{\coeff} u = -\sum_{ij=1}^{d} \pa_i a_{ij} \pa_j u +
  \frac{c}{r^{2}_{\dom}} u$, $c \ge 0$, be a strongly elliptic
  operator (so $b_i = 0$). 
  In case $p \in \maV \subset \pa \dom$ is a vertex that belongs to two
  adjacent Neumann edges, we assume that $c(p) > 0$.
  Then $P^{\coeff}$ is coercive. Moreover, 
\begin{equation}
 P^{\coeff} : \Vm \ede \maK_{a+1}^{m+1}(\dom) \cap \{u\vert_{\pa_D \dom}
 = 0 \} \ \to \ \Vmm \ede \maK_{a-1}^{m-1}(\dom) \oplus
 \maK_{a-1/2}^{m-1/2}(\dom)
\end{equation}
is an isomorphism and its inverse has norm
\begin{equation*}
  \|(P^{\coeff})^{-1}\| \ \le \ C \,\rho(\beta)^{-N_m - 1} 
  \big (\sum_{ij} \|a_{ij}\|_{\maW^{m, \infty}(\dom)} + 
  \|c\|_{\maW^{m, \infty}(\dom)} \big)^{N_m} \,,\quad |a| \le 1 \,,
\end{equation*}
with $N_m$ as in Theorem \ref{thm.main2} and $C$ independent of 
$\coeff$.
\end{theorem}

% \bibliographystyle{abbrv}
% \bibliography{paramReg}

\begin{thebibliography}{10}

\bibitem{ADN59}
S.~Agmon, A.~Douglis, and L.~Nirenberg.
\newblock Estimates near the boundary for solutions of elliptic partial
  differential equations satisfying general boundary conditions. {I}.
\newblock {\em Comm. Pure Appl. Math.}, 12:623--727, 1959.

\bibitem{sobolev}
B.~Ammann, A.~D. Ionescu, and V.~Nistor.
\newblock Sobolev spaces on {L}ie manifolds and regularity for polyhedral
  domains.
\newblock {\em Doc. Math.}, 11:161--206 (electronic), 2006.

\bibitem{ammannNistorWReg}
B.~Ammann and V.~Nistor.
\newblock Weighted {S}obolev spaces and regularity for polyhedral domains.
\newblock {\em Comput. Methods Appl. Mech. Engrg.}, 196(37-40):3650--3659,
  2007.

\bibitem{ApelNicaise98}
T.~Apel and S.~Nicaise.
\newblock The finite element method with anisotropic mesh grading for elliptic
  problems in domains with corners and edges.
\newblock {\em Math. Methods Appl. Sci.}, 21(6):519--549, 1998.

\bibitem{BKP79}
I.~Babu{\v{s}}ka, R.~B. Kellogg, and J.~Pitk{\"a}ranta.
\newblock Direct and inverse error estimates for finite elements with mesh
  refinements.
\newblock {\em Numer. Math.}, 33(4):447--471, 1979.

\bibitem{BBX2003a}
C.~Bacuta, J.~H. Bramble, and J.~Xu.
\newblock Regularity estimates for elliptic boundary value problems in {B}esov
  spaces.
\newblock {\em Math. Comp.}, 72(244):1577--1595, 2003.

\bibitem{BBX2003b}
C.~Bacuta, J.~H. Bramble, and J.~Xu.
\newblock Regularity estimates for elliptic boundary value problems with smooth
  data on polygonal domains.
\newblock {\em J. Numer. Math.}, 11(2):75--94, 2003.

\bibitem{BMNZ}
C.~Bacuta, A.~Mazzucato, V.~Nistor, and L.~Zikatanov.
\newblock Interface and mixed boundary value problems on {$n$}-dimensional
  polyhedral domains.
\newblock {\em Doc. Math.}, 15:687--745, 2010.

\bibitem{BNZ3D1}
C.~Bacuta, V.~Nistor, and L.~Zikatanov.
\newblock Improving the rate of convergence of high-order finite elements on
  polyhedra. {I}. {A} priori estimates.
\newblock {\em Numer. Funct. Anal. Optim.}, 26(6):613--639, 2005.

\bibitem{BNZ3D2}
C.~Bacuta, V.~Nistor, and L.~Zikatanov.
\newblock Improving the rate of convergence of high-order finite elements on
  polyhedra. {II}. {M}esh refinements and interpolation.
\newblock {\em Numer. Funct. Anal. Optim.}, 28(7-8):775--824, 2007.

\bibitem{BZ01}
J.~H. Bramble and X.~Zhang.
\newblock Uniform convergence of the multigrid {$V$}-cycle for an anisotropic
  problem.
\newblock {\em Math. Comp.}, 70(234):453--470, 2001.

\bibitem{BrezisBookFA}
H.~Brezis.
\newblock {\em Functional analysis, {S}obolev spaces and partial differential
  equations}.
\newblock Universitext. Springer, New York, 2011.

\bibitem{MR2988781}
C.~Carvalho and Y.~Qiao.
\newblock Layer potentials {$C^*$}-algebras of domains with conical points.
\newblock {\em Cent. Eur. J. Math.}, 11(1):27--54, 2013.

\bibitem{Ciarlet78}
P.~Ciarlet.
\newblock {\em The Finite Element Method for Elliptic Problems}, volume~4 of
  {\em Studies in Mathematics and Its Applications}.
\newblock North-Holland, Amsterdam, 1978.

\bibitem{CDS2011}
A.~Cohen, R.~Devore, and C.~Schwab.
\newblock Analytic regularity and polynomial approximation of parametric and
  stochastic elliptic {PDE}'s.
\newblock {\em Anal. Appl. (Singap.)}, 9(1):11--47, 2011.

\bibitem{CDNBook}
M.~Costabel, M.~Dauge, and S.~Nicaise.
\newblock Corner singularities and analytic regularity for linear elliptic
  systems. part i: Smooth domains.
\newblock Hal Preprint Archive, hal-00453934v2.

\bibitem{CDN12}
M.~Costabel, M.~Dauge, and S.~Nicaise.
\newblock Analytic regularity for linear elliptic systems in polygons and
  polyhedra.
\newblock {\em Math. Models Methods Appl. Sci.}, 22(8):1250015, 63, 2012.

\bibitem{CostabelStephan85}
M.~Costabel and E.~Stephan.
\newblock Boundary integral equations for mixed boundary value problems in
  polygonal domains and {G}alerkin approximation.
\newblock In {\em Mathematical models and methods in mechanics}, volume~15 of
  {\em Banach Center Publ.}, pages 175--251. PWN, Warsaw, 1985.

\bibitem{dauge1982}
M.~Dauge.
\newblock Op\'erateur de {S}tokes dans des espaces de {S}obolev \`a poids sur
  des domaines anguleux.
\newblock {\em Canad. J. Math.}, 34(4):853--882, 1982.

\bibitem{daugeBook}
M.~Dauge.
\newblock {\em Elliptic boundary value problems on corner domains}, volume 1341
  of {\em Lecture Notes in Mathematics}.
\newblock Springer-Verlag, Berlin, 1988.
\newblock Smoothness and asymptotics of solutions.

\bibitem{DineenBook}
S.~Dineen.
\newblock {\em Complex analysis on infinite-dimensional spaces}.
\newblock Springer Monographs in Mathematics. Springer-Verlag London, Ltd.,
  London, 1999.

\bibitem{MR1658089}
E.~Fabes, O.~Mendez, and M.~Mitrea.
\newblock Boundary layers on {S}obolev-{B}esov spaces and {P}oisson's equation
  for the {L}aplacian in {L}ipschitz domains.
\newblock {\em J. Funct. Anal.}, 159(2):323--368, 1998.

\bibitem{TrudingerBook}
D.~Gilbarg and N.~Trudinger.
\newblock {\em Elliptic partial differential equations of second order}.
\newblock Classics in Mathematics. Springer-Verlag, Berlin, 2001.
\newblock Reprint of the 1998 edition.

\bibitem{MR3440270}
W.~He, Z.~Zhang, and R.~Zhao.
\newblock The highest superconvergence of the tri-linear element for
  {S}chr\"odinger operator with singularity.
\newblock {\em J. Sci. Comput.}, 66(1):1--18, 2016.

\bibitem{hormander3}
L.~H{\"o}rmander.
\newblock {\em The analysis of linear partial differential operators. {III}}.
\newblock Classics in Mathematics. Springer, Berlin, 2007.
\newblock Pseudo-differential operators, Reprint of the 1994 edition.

\bibitem{MR3116171}
M.~Kohr, C.~Pintea, and W.~L. Wendland.
\newblock Layer potential analysis for pseudodifferential matrix operators in
  {L}ipschitz domains on compact {R}iemannian manifolds: applications to
  pseudodifferential {B}rinkman operators.
\newblock {\em Int. Math. Res. Not. IMRN}, (19):4499--4588, 2013.

\bibitem{Kondratiev67}
V.~A. Kondrat$'$ev.
\newblock Boundary value problems for elliptic equations in domains with
  conical or angular points.
\newblock {\em Transl. Moscow Math. Soc.}, 16:227--313, 1967.

\bibitem{KMR}
V.~Kozlov, V.~Maz$'$ya, and J.~Rossmann.
\newblock {\em Spectral problems associated with corner singularities of
  solutions to elliptic equations}, volume~85 of {\em Mathematical Surveys and
  Monographs}.
\newblock American Mathematical Society, Providence, RI, 2001.

\bibitem{liMN10}
H.~Li, A.~Mazzucato, and V.~Nistor.
\newblock Analysis of the finite element method for transmission/mixed boundary
  value problems on general polygonal domains.
\newblock {\em Electron. Trans. Numer. Anal.}, 37:41--69, 2010.

\bibitem{liN09}
H.~Li and V.~Nistor.
\newblock Analysis of a modified {S}chr\"odinger operator in 2{D}: regularity,
  index, and {FEM}.
\newblock {\em J. Comput. Appl. Math.}, 224(1):320--338, 2009.

\bibitem{HNParam}
H.~Li, V.~Nistor, and Y.~Qiao.
\newblock Uniform shift estimates for transmission problems and optimal rates
  of convergence for the parametric finite element method.
\newblock In {\em Numerical analysis and its applications}, volume 8236 of {\em
  Lecture Notes in Comput. Sci.}, pages 12--23. Springer, Heidelberg, 2013.

\bibitem{LionsMagenes1}
J.-L. Lions and E.~Magenes.
\newblock {\em Non-homogeneous boundary value problems and applications. {V}ol.
  {I}}.
\newblock Springer-Verlag, New York-Heidelberg, 1972.
\newblock Translated from the French by P. Kenneth, Die Grundlehren der
  mathematischen Wissenschaften, Band 181.

\bibitem{elast}
A.~Mazzucato and V.~Nistor.
\newblock Well-posedness and regularity for the elasticity equation with mixed
  boundary conditions on polyhedral domains and domains with cracks.
\newblock {\em Arch. Ration. Mech. Anal.}, 195(1):25--73, 2010.

\bibitem{MR1986694}
D.~Mitrea and I.~Mitrea.
\newblock On the {B}esov regularity of conformal maps and layer potentials on
  nonsmooth domains.
\newblock {\em J. Funct. Anal.}, 201(2):380--429, 2003.

\bibitem{MonkBook}
P.~Monk.
\newblock {\em Finite element methods for {M}axwell's equations}.
\newblock Numerical Mathematics and Scientific Computation. Oxford University
  Press, New York, 2003.

\bibitem{NP}
S.~A. Nazarov and B.~A. Plamenevsky.
\newblock {\em Elliptic problems in domains with piecewise smooth boundaries},
  volume~13 of {\em de Gruyter Expositions in Mathematics}.
\newblock Walter de Gruyter \& Co., Berlin, 1994.

\bibitem{NistorSchwab1}
V.~Nistor and C.~Schwab.
\newblock High-order {G}alerkin approximations for parametric second-order
  elliptic partial differential equations.
\newblock {\em Math. Models Methods Appl. Sci.}, 23(9):1729--1760, 2013.

\bibitem{LQ16}
Y.~Qiao and H.~Li.
\newblock Double layer potentials on polygons and pseudodifferential operators
  on {L}ie groupoids.
\newblock submitted.

\bibitem{MR3300043}
V.~A. Rukavishnikov and E.~I. Rukavishnikova.
\newblock On the isomorphic mapping of weighted spaces by an elliptic operator
  with degeneration on the domain boundary.
\newblock {\em Differ. Equ.}, 50(3):345--351, 2014.
\newblock Translation of Differ. Uravn. {{\bf{5}}0} (2014), no. 3, 349--355.

\end{thebibliography}
% \end{document}

\end{document}